\documentclass[preprint,12pt]{elsarticle}

\usepackage{pgf,tikz}
\usetikzlibrary{arrows}
\usepackage{array,multirow,makecell}
\setcellgapes{1pt} \makegapedcells
\usepackage{slashbox}
\usepackage{amssymb}
\usepackage{subfig}
\usepackage[normalem]{ ulem }
\usepackage{soul}
\usepackage{graphicx}
\usepackage{here}
\usepackage{epsfig}
\usepackage{color}
\usepackage{pgf,tikz,amsthm}
\usetikzlibrary{arrows}
\usepackage{epstopdf}

\usepackage{hyperref}
\hypersetup{
    colorlinks=true,
    linkcolor=blue,
    filecolor=magenta,      
    urlcolor=cyan,
    citecolor=red
}

\newtheorem{theorem}{\textbf{Theorem}}
\newtheorem{lemma}{\textbf{Lemma}}
\newtheorem{proposition}{\textbf{Proposition}}

\newtheorem{remark}{\textbf{Remark}}

\newcommand{\R}{\mathcal{R}}
\newcommand{\comment}[1]{}

\newcommand{\tcm}{\textcolor{magenta}}

\begin{document}

\begin{frontmatter}



\title{Sustainable vector/pest control using the permanent Sterile Insect Technique }

\author[label1]{R. Anguelov}
\ead{roumen.anguelov@up.ac.za}
\author[label1,label2,label3]{Y. Dumont}
\ead{yves.dumont@cirad.fr; yves.dumont@up.ac.za}
\author[label1]{I.V. Yatat Djeumen}
\ead{ivric.yatatdjeumen@up.ac.za}

\address[label1]{University of Pretoria, Department of Mathematics and Applied Mathematics, Pretoria, South Africa}
 \address[label2]{UMR AMAP, F-34398 Montpellier, France}
 \address[label3]{AMAP, Univ Montpellier, CIRAD, CNRS, INRA, IRD, Montpellier, France}

\begin{abstract}
Vector/Pest control is essential to reduce the risk of vector-borne diseases or losses in crop fields. Among biological control tools, the sterile insect technique (SIT), is the most promising one. SIT control generally consists of massive releases of sterile insects in the targeted area in order to reach elimination or to lower the pest population under a certain threshold. 
The models presented here are minimalistic with respect to the number of parameters and variables. The first model deals with the dynamics of the vector population while the second
model, the SIT model, tackles the interaction between treated males and wild female vectors. 
For the vector population model, the elimination equilibrium \textbf{0} is globally asymptotically stable when the basic offspring number,  $\mathcal{R}$, is lower or equal to one, whereas \textbf{0} becomes unstable and one stable positive equilibrium exists, with well determined basins of attraction, when $\mathcal{R}>1$. For
the SIT model, we obtain a threshold number of treated male vectors above which the control of wild female vectors is effective: the massive release control. When the amount of treated male vectors is lower than the aforementioned threshold number, the SIT model experiences a bistable situation involving the elimination equilibrium and a positive equilibrium.  However, practically, massive releases of sterile males are only possible for a short period of time. That is why, 
using the bistability property, we develop a new strategy to maintain the wild population under a certain threshold, for a permanent and sustainable low level of SIT control. We illustrate our theoretical results with numerical simulations, in the case of SIT mosquito control.
\end{abstract}


\end{frontmatter}



\section{Introduction}
In the last decades, the development of sustainable vector control methods has become one of the most challenging issue to reduce the impact of human vector borne diseases, like malaria, dengue, chikungunya or crop pests, like fruit flies.

Several control techniques have been developed or are under development. However, the process to reach field applications is long and complex. Modeling, and in particular Mathematical Modeling has become a useful tool in Human Epidemiology since the pioneering works of Sir R. Ross and his malaria model \cite{Ross1910, Ross1911}. Numerous models have been developed to understand the dynamics of diseases and pests to test "in silico" the usefulness or not of control strategies (and their combination).

In this paper, we focus on the Sterile Insect Technique (SIT). This is an old control techniques that have been used more or less successfully on the field against various kind of Pests or Vectors (see \cite{SIT} for various examples). The classical SIT consists of mass releases of males sterilized by ionizing radiation. The released sterile males transfer their sterile sperms to wild females, which results in a progressive decay of the targeted population. For mosquitoes, other sterilization techniques have been developed using either genetics (the RIDL technique) or bacteria (wolbachia) \cite{Sinkin2004}. For fruit flies, only ionizing radiation has been used, so far \cite{SIT}.

Our work is a companion paper of \cite{Strugarek2019}, where SIT against mosquitoes only has been considered. However, it is important to notice that the results obtained in \cite{Strugarek2019} can be used against crop Pest too. An important assumption in \cite{Strugarek2019} is that the insect population dynamics exhibit a strong Allee effect. Then, the application of SIT for an estimated finite time is sufficient to drive the population below the minimum survival density. However, for insect population the minimum survival density tends to be very close to extinction, that is an area of the domain, where deterministic modelling is not considered adequate. Hence, in this paper we do not make such assumption, but rather propose control which relies on Allee effect generated by the SIT control. Indeed, in previous works, e.g. \cite{Anguelov2012TIS, Dumont2012, Dumont2010}, it has been shown that even low levels of SIT control produce a tangible minimum survival density, below which the population declines to extinction.  In this setting, if we need to keep the insect population below certain level and/or to sustain the decay, the SIT cannot be discontinued. In this sense we talk about "permanent" SIT. The level of permanent SIT control is determined by available resources. Once this level is known, higher level of releases can be used in short term in order to bring the insect population density below the minimum survival level associated with the lower, but long-term sustainable SIT level of control.   The aim of this paper is to show the feasibility of this type of SIT control strategy as well as specific methods for calculating its essential parameters.

The outline of the paper is as follows. In the next section, we present a minimalistic entomological model
of wild insect population and the discussion of its global dynamical
properties. Section \ref{SIT-continuous} deals with the study of the
SIT mathematical model in the case of constant and continuous SIT
releases. The key finding is the identification of a threshold
number of sterile male vectors above which the control of the wild population is effective; that is, the wild population declines to extinction. Section \ref{caracterisation} is devoted to the
characterization of the minimal time necessary to reduce the amount
of wild vector population under a given threshold when using SIT
releases, that is, by considering the SIT model studied in section
\ref{SIT-continuous}. Section \ref{SIT-Pulse} deals with the study of the
SIT mathematical model in the case of periodic and impulsive SIT
releases. Notably, by using suitable comparison arguments, we
provide condition of reaching elimination of wild vector population
with periodic and pulse SIT releases and, characterize the minimal
time necessary to lower the wild vector population under a given
threshold in order to reduce the epidemiological risk. The
theoretical results are discussed and supported by numerical
simulations in section \ref{Numerical-simulations}. Concluding
remarks as to show how this work fits in the literature and can be
extended are given in section \ref{conclusion}. 

\section{A minimalistic entomological model}\label{SIT-minimalistic}
The model presented in this section is minimalistic in the sense that it uses smallest possible number of compartments which allows for adequate modelling of the mechanism of SIT control. It is simpler than the models in \cite{Anguelov2012TIS} and \cite{Dumont2010}. Nevertheless, and we will see in the sequel, it has the same asymptotic properties as the other mentioned models. The advantages of using this simpler model are two fold: On the one hand, while the model remains biologically accurate, it allows for a complete analysis to be carried out. On the other hand, it is more generic and can be applied to a variety of insect populations.  

The model is given as a system of ordinary differential equations as follows:
\begin{equation}\label{Mosquitoes-ode}
 \left\{%
\begin{array}{lcl}
 \displaystyle\frac{dA}{d t} &=& \phi F-(\gamma+\mu_{A,1}+\mu_{A,2}A)A,\\
    \displaystyle\frac{dM}{d t} &=& (1-r)\gamma A-\mu_MM,\\
    \displaystyle\frac{dF}{d t} &=& r\gamma A-\mu_FF,\\
\end{array}
\right.
\end{equation}
where the parameters and state variables are described in Table
\ref{table-description}. 

\begin{table}[H]
  \centering
  \begin{tabular}{|l|l|}
    \hline
    Symbol & Description \\
    \hline
    $A$ & Aquatic stage (gathering eggs, larvae, nymph stages) \\
    \hline
    $F$ & Fertilized and eggs-laying females \\
    \hline
    $M$ & Males \\
    \hline\hline
    $\phi$ & Number of eggs at each deposit per capita (per day) \\
    \hline
    $\gamma$ & Maturation rate from larvae to adult (per day) \\
    \hline
    $\mu_{A,1}$ & Density independent mortality rate of the aquatic stage (per day) \\
    \hline
    $\mu_{A,2}$ & Density dependent mortality rate of the aquatic stage (per day$\times$ number) \\
    \hline
    $r$ & Sex ratio \\
     \hline
    $1/\mu_F$ & Average lifespan of female (in days) \\
    \hline
    $1/\mu_M$ & Average lifespan of male (in days) \\
    \hline
  \end{tabular}
  \caption{Description of parameters and state variables of model (\ref{Mosquitoes-ode})}\label{table-description}
\end{table}
Contrary to \cite{Strugarek2019}, we assume a density-dependent mortality rate in the aquatic stage. This may correspond to an intra-specific competition between the larvae stages, for instance. However, the forthcoming methodology could be applied for a system where the non-linearity stands for the birth-rate, like in \cite{Strugarek2019,Dumont2012,DT2016}.

We set $x = (A,M,F)'$ and
$\mathcal{D}=\mathbb{R}^3_+=\{x\in\mathbb{R}^3:x\geq\textbf{0}\}$. Then model (\ref{Mosquitoes-ode}) can be written in the form 
\begin{equation}\label{ODE-system}
\frac{dx}{dt}=f(x),
\end{equation}
where $f:\mathbb{R}^3\to\mathbb{R}^3$ represents the right hand side of (\ref{Mosquitoes-ode}). Function $f$ is continuous and continuously differentiable on $\mathbb{R}^3$. Thus, according to \cite[Theorem III.10.VI]{Walter1998}, for any initial condition a unique solution exists, at least locally. The vector field defined by $f$ is either tangential or directed inwards on $\partial D$. Therefore, for any initial condition in $\mathcal{D}$ the solution of (\ref{ODE-system}) remains in $\mathcal{D}$ for its maximal interval of existence \cite[Theorem III.10.XVI]{Walter1998}. In the sequel we consider the vector population model in the form (\ref{Mosquitoes-ode}) or in the form (\ref{ODE-system}) on the domain $\mathcal{D}$. In order to obtain existence of the solutions in $\mathcal{D}$, it is sufficient to obtain a priori upper bounds. This can be done as follows. 

We observe that system (\ref{Mosquitoes-ode}) is monotone \cite[Proposition 3.1.1]{Smith2008}. Indeed, for any $x\in\mathcal{D}$ the Jacobian  
\begin{equation}\label{Jacobian}
J(x)=\left(
    \begin{array}{ccc}
      -(\gamma+\mu_{A,1})-2\mu_{A,2}A & 0 &\phi \\
      (1-r)\gamma & -\mu_M & 0\\
      r\gamma & 0 &-\mu_F \\
    \end{array}
  \right)
\end{equation}
is a Metzler matrix, i.e. all its off diagonal elements are non-negative. 
The inequality
\begin{equation}\label{inequality-ode}
r\gamma-\displaystyle\frac{\mu_F}{2\phi}(\gamma+\mu_{A,1}+\mu_{A,2}A)<0
\end{equation}
holds for all sufficiently large $A$. Let $m>0$ and let $A_m$ be so
large that in addition to (\ref{inequality-ode}) the following
inequalities also hold:
\begin{equation}\label{bm}
\begin{array}{cclc}
  A_m & \geq & m, &  \\
  F_m & := & \displaystyle\frac{(\gamma+\mu_{A,1}+\mu_{A,2}A_m)A_m}{2\phi} & \geq m, \\
  M_m & := & \displaystyle\frac{2(1-r)\gamma A_m}{\mu_M} & \geq m.
\end{array}
\end{equation}
For every $m>0$ let 
\begin{equation}\label{defbm}
b_m=(A_m,M_m,F_m)'
\end{equation}
be a vector with coordinates satisfying (\ref{inequality-ode}) and (\ref{bm}). 
Then
\begin{equation}\label{fbm}
   f(b_m)=\left(
           \begin{array}{c}
             -\phi F_m \\
             -(1-r)\gamma A_m \\
             A_m(r\gamma-\displaystyle\frac{\mu_F}{2\phi}(\gamma+\mu_{A,1}+\mu_{A,2}A_m)) \\
           \end{array}
         \right)<\textbf{0}.
\end{equation}
Using \cite[Proposition 3.2.1]{Smith2008}, the solution initiated at $b_m$ is decreasing. Then, using again the monotonicity of the system, see \cite[Proposition 3.2.1]{Smith2008}, for any solution of (\ref{Mosquitoes-ode}) initiated in $\mathcal{D}$ we have 
\begin{equation}\label{upperbound}
x(t)\leq b_{||x(0)||_\infty}.
\end{equation}
The a priori upper bound given in (\ref{upperbound}) provides for existence of the solution for all $t\geq 0$. Therefore, (\ref{Mosquitoes-ode}) defines a dynamical system on $\mathcal{D}$.
\par

The stability properties of the extinction equilibrium $\textbf{0}=(0,0,0)'$ are usually described in terms of the basic offspring number $\R$ of the population, i.e. the self-reproduction of an individual (number of females produced by a single female) during its lifetime, assuming that the population is so small that the density dependent mortality can be ignored. The basic offspring number related to model (\ref{Mosquitoes-ode}) is defined as follows
\begin{equation}
\R=\displaystyle\frac{r\gamma\phi}{\mu_F(\gamma+\mu_{A,1})}.
\label{BO}
\end{equation}
The Jacobian of system (\ref{Mosquitoes-ode}) computed at the extinction equilibrium is 
\begin{equation}\label{JacobianZero}
J(\textbf{0})=\left(
    \begin{array}{ccc}
      -(\gamma+\mu_{A,1}) & 0 &\phi \\
      (1-r)\gamma & -\mu_M & 0\\
      r\gamma & 0 &-\mu_F \\
    \end{array}
  \right).    
\end{equation}
Its eigenvalues are $-\mu_M$ and the roots of the equation
\begin{equation}\label{charequation}
    \lambda^2+(\gamma+\mu_{A,1}+\mu_F)\lambda+(\gamma+\mu_{A,1})\mu_F(1-\R)=0.
\end{equation}
It is easy to see that if $\R<1$, all eigenvalues of $J(\textbf{0})$ are either negative or have negative real parts, that is $\textbf{0}$ is asymptotically stable. If $\R>1$, the Jacobian has two negative eigenvalues and a positive one. Hence, $\textbf{0}$ is unstable.

The existence of an endemic equilibrium also depends on the value of $\R$. Setting the right hand side of (\ref{Mosquitoes-ode}) to zero we obtain the equilibrium $\textbf{0}$ and the equilibrium $E^*=(A^*,M^*,F^*)'$ given by 
\begin{equation}\label{Wild-equilibria-definition}
 \left\{%
\begin{array}{rcl}
A^*&=& \displaystyle\frac{(\gamma+\mu_{A,1})}{\mu_{A,2}}(\R-1),\\
 M^*&=& \displaystyle\frac{(1-r)\gamma A^*}{\mu_M},\\
 F^*&=& \displaystyle\frac{r\gamma A^*}{\mu_F}.\\
\end{array}
\right.
\end{equation}
Clearly, $E^*\in\mathcal{D}$ and $E^*\neq\textbf{0}$ if and only if $\R>1$. 
We summarize these results with some more details related to basins of attraction of equilibria in the following theorem. 

\begin{theorem}\label{Mosquitoes-ode-theorem}
Model (\ref{Mosquitoes-ode}) defines a forward dynamical system on 
$\mathcal{D}$. Furthermore, 
\begin{enumerate}
    \item[1)] If $\R\leq1$ then $\bf{0}$ is globally asymptotically stable on
    $\mathcal{D}$.
    \item[2)] If $\R>1$ then $E^*$ is stable with
    basin of attraction $$\mathcal{D}\setminus\{x=(A,M,F)'\in\mathbb{R}^3_+:A=F=0\},$$  and $\bf{0}$ is unstable with the non negative $M-$axis being a stable manifold.
\end{enumerate}
\end{theorem}
\begin{proof} As mentioned, it remains to prove the statements regarding the basins of attraction. We use an approach similar to the approach in \cite{Anguelov2019} for the analysis of bi-stable monotone systems. 
1) Let $\R\leq 1$. Let $x=x(t)$ be any solution initiated in $\mathcal{D}$. Denote by $y=y(t)$ the solution of (\ref{Mosquitoes-ode}) with initial condition $y(0)=b_{||x(0)||_{\infty}}$. It follows from the inequality (\ref{fbm}) that the function $y$ is decreasing and, therefore, it converges. The limit is necessarily an equilibrium (see also \cite[page 35]{Smith2008}). Considering that there is only one equilibrium in $\mathcal{D}$, we conclude that $\lim\limits_{t\rightarrow+\infty} y(t)=\textbf{0}$. Using that (\ref{Mosquitoes-ode}) is a monotone system, the inequalities $\textbf{0}\leq x(0)\leq b_{||x(0)||_{\infty}}$, we have
$$
\textbf{0}\leq x(t)\leq y(t), \ \ t\geq 0.
$$
Therefore, $\lim\limits_{t\rightarrow+\infty} x(t)=\textbf{0}$, which proves the global asymptotic stability of $\textbf{0}$ on $\mathcal{D}$.

2) To prove the stability and basin of attraction we use  \cite[Theorem 2.2.2]{Smith2008}. This theorem applies to strongly monotone systems. We recall that if the Jacobian of $f$ is a Metzler irreducible matrix for every $x\in\mathcal{D}$, then (\ref{ODE-system}) is strongly monotone \cite[Theorem 4.1.1]{Smith2008}. The Jacobian (\ref{Jacobian}) associated with (\ref{Mosquitoes-ode}) is not irreducible, since the equation for $M$ can be decoupled. We consider the subsystem for $A$ and $F$, that is,
\begin{equation}\label{Mosquitoes-odeAF}
 \left\{%
\begin{array}{lcl}
 \displaystyle\frac{dA}{d t} &=& \phi F-(\gamma+\mu_{A,1}+\mu_{A,2}A)A,\\
\displaystyle\frac{dF}{d t} &=& r\gamma A-\mu_FF,\\
\end{array}
\right.
\end{equation}
which defines a dynamical system on $\mathbb{R}^2_+$. The Jacobian 
\begin{equation}\label{JacobianAF}
\tilde{J}(A,F)=\left(
    \begin{array}{cc}
      -(\gamma+\mu_{A,1})-2\mu_{A,2}A & \phi \\
       r\gamma  &-\mu_F \\
    \end{array}
  \right)
\end{equation}
is clearly irreducible. We apply \cite[Theorem 2.2.2]{Smith2008} to the two dimensional interval 
$$
\{(A,F)'\in\mathbb{R}^2_+:0\leq A\leq A^*,\,0\leq F\leq F^*\}.
$$
It follows that, all solutions initiated in this interval, excluding the end points, converge either all to $(0,0)'$ or all to $(A^*,F^*)'$. The characteristic equation of $\tilde{J}(0,0)$ is exactly (\ref{charequation}), which produces one positive and one negative root. Considering that $\tilde{J}(0,0)$ is a Metzler matrix, it has a strictly positive eigenvector corresponding to the positive eigenvalue. Hence, it is not possible that all solutions converge to $(0,0)'$. Therefore, they all converge to $(A^*,F^*)'$. The implication for the three dimensional system (\ref{Mosquitoes-ode}) is that all solutions initiated in the interval $[\textbf{0},E^*]$, excluding the $M$-axis, converge to $E^*$. 

Using similar argument as in 1), any solution initiated at a point larger than $E^*$ converges to $E^*$. Since any point in $\mathcal{D}\setminus\{x=(A,M,F)'\in\mathbb{R}^3_+:A=F=0\}$ can be placed between a point below $E^*$, but not on $M$-axis and a point above $E^*$, all solutions initiated in $\mathcal{D}\setminus\{x=(A,M,F)'\in\mathbb{R}^3_+:A=F=0\}$ converge to $E^*$. The monotone convergence of the solutions initiated below and above $E^*$ implies the asymptotic stability of $E^*$ as well. The basin of attraction cannot be extended further, since the nonnegative $M$-axis is the attractive manifold corresponding to the eigenvalue $-\mu_M$ of $J(\textbf{0})$. 
\end{proof}

\section{The SIT model in the case of constant and continuous releases}\label{SIT-continuous}
 In the sequel, we assume that $\R>1$.
We take into account the constant release of sterile male vectors
$M_T$ by adding to model (\ref{Mosquitoes-ode}) an equation for
$M_T$.  Altogether, the SIT model reads as
\begin{equation}\label{SIT-ode-original}
 \left\{%
\begin{array}{lcl}
 \displaystyle\frac{dA}{d t} &=& \phi F-(\gamma+\mu_{A,1}+\mu_{A,2}A)A,\\
     \displaystyle\frac{dM}{d t} &=& (1-r)\gamma A-\mu_MM,\\
    \displaystyle\frac{dF}{d t} &=& \displaystyle\frac{ M}{M+M_T}r\gamma A-\mu_FF,\\
    \displaystyle\frac{dM_T}{d t} &=& \Lambda-\mu_TM_T.\\
\end{array}
\right.
\end{equation}
The quantity $\Lambda$ is the number of sterile insects released per
unit of time. Assuming $t$ large enough, we may assume that $M_T(t)$ has reached its equilibrium value $M_T^*:=\Lambda/\mu_T$. 
Thus, model (\ref{SIT-ode-original}) reduces to
\begin{equation}\label{SIT-ode}
 \left\{%
\begin{array}{lcl}
 \displaystyle\frac{dA}{d t} &=& \phi F-(\gamma+\mu_{A,1}+\mu_{A,2}A)A,\\
     \displaystyle\frac{dM}{d t} &=& (1-r)\gamma A-\mu_MM,\\
    \displaystyle\frac{dF}{d t} &=& \displaystyle\frac{ M}{M+M_T^*}r\gamma A-\mu_FF,\\
\end{array}
\right.
\end{equation}
where parameters and state variables are described in Table
\ref{table-description}. Model (\ref{SIT-ode}) defines a monotone dynamical system on $\mathcal{D}$.

\subsection{Equilibria of the SIT model (\ref{SIT-ode}): existence and stability}
  Equilibria of the SIT model (\ref{SIT-ode}) are obtained by
solving the system
\begin{equation}\label{ODE-TIS-equilibre}
\left\{%
\begin{array}{rcl}
  \phi F-(\gamma+\mu_{A,1}+\mu_{A,2}A)A &=&0,\\
 (1-r)\gamma A-\mu_MM &=&0,\\
  \displaystyle\frac{ M}{M+M_T^*}r\gamma A-\mu_FF &=&0.\\
\end{array}
\right.
\end{equation}
From (\ref{ODE-TIS-equilibre})$_1$ and (\ref{ODE-TIS-equilibre})$_2$
we have
\begin{equation}\label{ODE-TIS-A}
   A=\displaystyle\frac{\mu_M}{(1-r)\gamma}M
\end{equation}
 and
 \begin{equation}\label{ODE-TIS-F}
    F=\displaystyle\frac{(\gamma+\mu_{A,1}+\mu_{A,2}A)A}{\phi}=\displaystyle\frac{(\gamma+\mu_{A,1})}{\phi}\frac{\mu_M}{(1-r)\gamma}M+\displaystyle\frac{\mu_{A,2}}{\phi}\left(\displaystyle\frac{\mu_M}{(1-r)\gamma}M\right)^2.
 \end{equation}
 Substituting in (\ref{ODE-TIS-equilibre})$_3$ leads to $M=0$ or

 \begin{equation}\label{dagg}
   \displaystyle\frac{ r\gamma M}{M+M_T^*}-\frac{\mu_F(\gamma+\mu_{A,1})}{\phi}-\frac{\mu_F\mu_{A,2}}{\phi}\frac{\mu_M}{(1-r)\gamma}M=0.
 \end{equation}
Let us set $\alpha=M_T/M$ then in term of $\alpha$, equation
(\ref{dagg}) can be written as 
\begin{equation}\label{dagg2}
    \alpha^2-a_1\alpha+a_0=0,
\end{equation}
 where
$$
\begin{array}{ccl}
  a_1 & = & \displaystyle\frac{r\gamma\phi}{\mu_F(\gamma+\mu_{A,1})}-1-\displaystyle\frac{\mu_{A,2}\mu_M}{(\gamma+\mu_{A,1})(1-r)\gamma}M_T^*, \\
  a_0 & = & \displaystyle\frac{\mu_{A,2}\mu_M}{(\gamma+\mu_{A,1})(1-r)\gamma}M_T^*.
\end{array}
$$
Setting
  $Q = \displaystyle\frac{\mu_{A,2}\mu_M}{(\gamma+\mu_{A,1})(1-r)\gamma}$,
 (\ref{dagg2}) assumes the form

\begin{equation}\label{dagg3}
    \alpha^2-(\R-1-QM_T^*)\alpha+QM_T^*=0.
\end{equation}
The discriminant of (\ref{dagg3}) is
$$
\Delta (M_T^*)= ((\sqrt{\R}-1)^2-M_T^*Q)((\sqrt{\R}+1)^2-M_T^*Q).
$$
The equation $\Delta (M_T^*)=0$ has two positive solutions $M_{T_1}$ and
$M_{T_2}$:

\begin{equation}\label{MT1}
    M_{T_{1}} = \displaystyle\frac{(\sqrt{\R}-1)^2}{Q}, \quad M_{T_{2}} = \displaystyle\frac{(\sqrt{\R}+1)^2}{Q}.
\end{equation}
 Then, we have several possible cases to study:
\begin{itemize}
    \item When $M_T^*<M_{T_1}$, $\Delta(M_T^*)>0$ and
    (\ref{dagg3}) has two positive solutions $\alpha_+$ and $\alpha_-$ 

\begin{equation}
    \label{alpha}
    \alpha_{\pm}=\displaystyle\frac{(\R-1-QM_T^*)\pm\sqrt{(\R+1-QM_T^*)^2-4R}}{2}>0,
\end{equation}    
    because $\R-1-QM_T^*>\R-1-QM_{T_1}=2(\sqrt{\R}-1)>0$.
    \item When $M_T^*=M_{T_1}$ then $\Delta(M_T^*)=0$ and
    (\ref{dagg3}) has only one real solution $\alpha_\dag$ such that
\begin{equation} \label{alpha_dag}
    \alpha_{\dag}=\displaystyle\frac{(\R-1-QM_T^*)}{2}>0.
\end{equation}    
\item When $M_T^*\geq M_{T_2}$ then $\Delta(M_T)\geq 0$ and
    (\ref{dagg3}) has one or two real roots which are negative 
    because $\R-1-QM_T^*\leq\R-1-QM_{T_2}=-2(1+\sqrt{\R})<0$.
\end{itemize}
These results can be summarized as follows.
\begin{proposition}\label{SIT-ode-equilibre} Let $M_{T_{1}}$ given by (\ref{MT1}).
\begin{enumerate}
    \item If $M_T^*\in(0;M_{T_1})$ then model (\ref{SIT-ode}) has
    two positive equilibria $E_{1,2}=(A_{1,2},M_{1,2},F_{1,2})'$
    with $E_1<E_2$ and
    $$
\begin{array}{ccl}
  A_{1,2} & = & \displaystyle\frac{\mu_M}{(1-r)\gamma}M_{1,2}, \\
  F_{1,2} & = & \displaystyle\frac{(\gamma+\mu_{A,1}+\mu_{A,2}A_{1,2})A_{1,2}}{\phi}, \\
  M_{1} & = & \displaystyle\frac{M_T}{\alpha_{+}},\\
  M_{2} & = & \displaystyle\frac{M_T}{\alpha_{-}},
\end{array}
    $$
    where $\alpha_{\pm}$ is given in (\ref{alpha}).
    \item If $M_T^*=M_{T_1}$ then model (\ref{SIT-ode}) has
    a positive equilibrium $E_{\dag}=(A_{\dag},M_{\dag},F_{\dag})'$
    where
    $$
\begin{array}{ccl}
  A_{\dag} & = & \displaystyle\frac{\mu_M}{(1-r)\gamma}M_{\dag}, \\
  F_{\dag} & = & \displaystyle\frac{(\gamma+\mu_{A,1}+\mu_{A,2}A_{\dag})A_{\dag}}{\phi}, \\
  M_{\dag} & = & \displaystyle\frac{M_T}{\alpha_\dag},
\end{array}
    $$
    where $\alpha_{\dag}$ is given by (\ref{alpha_dag}).
    \item If $M_T^*>M_{T_1}$ then model (\ref{SIT-ode}) has no positive equilibria.
\end{enumerate}
\end{proposition}

Before going further, let us make the following remark about the graphical analysis that also leads to a similar result as in Proposition \ref{SIT-ode-equilibre}. 
\begin{remark}\label{graphical-method}
Let us consider the following functions of $M$, defined on $\mathbb{R}^+$ by
$$
\begin{array}{l}
f_1(M)=\displaystyle\frac{M}{M+M_T^*},\\
f_2(M)=\displaystyle\frac{\mu_F(\gamma+\mu_{A,1})}{r\gamma\phi}+\frac{\mu_F\mu_{A,2}}{r\phi}\frac{\mu_M}{(1-r)\gamma^2}M.
\end{array}
$$
Hence, solve (\ref{dagg}) is equivalent to solve 
\begin{equation}\label{dagg-graphical}
    f_1(M)=f_2(M).
\end{equation}
Graphical analysis lead to the following three cases.
\begin{enumerate}
    \item Equation (\ref{dagg-graphical}) has zero solution. That is, the SIT model (\ref{SIT-ode}) does not have positive equilibrium.
    \item Equation (\ref{dagg-graphical}) has two positive solutions $M_1$ and $M_2$ with $M_1<M_2$. In that case, the SIT model (\ref{SIT-ode}) has two positive equilibria. In addition, by a direct comparison of the slopes of functions $f_1$ and $f_2$ at $M_{1,2}$ one deduces that:
\begin{equation}\label{slope-M1}
    \displaystyle\frac{M_T^*}{(M_1+M_T^*)^2}-\frac{\mu_F\mu_{A,2}}{r\phi}\frac{\mu_M}{(1-r)\gamma^2}>0
\end{equation}

and 

\begin{equation}\label{slope-M2}
    \displaystyle\frac{M_T^*}{(M_2+M_T^*)^2}-\frac{\mu_F\mu_{A,2}}{r\phi}\frac{\mu_M}{(1-r)\gamma^2}<0.
\end{equation}

\item Equation (\ref{dagg-graphical}) has one positive solution $M_\dag$ and, therefore, the SIT model (\ref{SIT-ode}) has also one positive equilibrium. In addition, it holds that
\begin{equation}\label{slope-Mdouble}
    \displaystyle\frac{M_T^*}{(M_\dag+M_T^*)^2}-\frac{\mu_F\mu_{A,2}}{r\phi}\frac{\mu_M}{(1-r)\gamma^2}=0.
\end{equation}
\end{enumerate}
\end{remark}
Remark \ref{graphical-method} will be helpful in the sequel.

\noindent The stability analysis is summarized in the following
\begin{theorem}
\label{SIT-ode-theorem}
System (\ref{SIT-ode}) defines a forward dynamical system on $\mathcal{D}$ for
any $M_T\in(0;+\infty)$. Moreover,
\begin{itemize}
    \item[(1)] If $M_T^*>M_{T_1}$ then equilibrium $\bf{0}$ is
    globally asymptotically stable on $\mathcal{D}$.
    \item[(2)] If $M_T^*=M_{T_1}$ then system (\ref{SIT-ode}) has two
    equilibria \textbf{\emph{0}} and $E_\dag$ with $\bf{0}<E_\dag$. The set $\{x\in \mathbb{R}^3:\textbf{\emph{0}}\leq
    x< E_\dag\}$ is in the basin of attraction of $\bf{0}$, while the set $\{x\in \mathbb{R}^3: x\geq E_\dag\}$ is in the basin of attraction of $E_\dag$.
    \item[(3)] If  $0<M_T^*<M_{T_1}$ then system (\ref{SIT-ode}) has three
    equilibria \textbf{\emph{0}}, $E_1$ and $E_2$ with $\bf{0}<E_1<E_2$. The set $\{x\in \mathbb{R}^3:\textbf{\emph{0}}\leq
    x< E_1\}$ is in the basin of attraction of \textbf{\emph{0}} while the set $\{x\in \mathbb{R}^3: x> E_1\}$ is in the basin of attraction
    of $E_2$.
\end{itemize}
\end{theorem}

\begin{proof}
Let us set $x=(A,M,F)'\in\mathcal{D}$ and $\Phi$ a vector-valued
function such that $\Phi(M_T^*,x)=f(x)$ where $f$ is the right hand
side of system (\ref{SIT-ode}). In compact form, we can therefore
write system (\ref{SIT-ode}) as follows:
\begin{equation}\label{SIT-ode-compact}
\displaystyle\frac{d x}{dt}=\Phi(M_T^*,x).
\end{equation}
Denote by $x_{M_T^*}(z, t)$ the solution of (\ref{SIT-ode-compact})
satisfying $x_{M_T^*}(z, 0) = z$. Consider the point $b_m$ as given by
(\ref{bm}). Using (\ref{fbm}) we have 

\begin{equation}\label{decreasing}
    \Phi(M_T^*,b_m)\leq
\Phi(0,b_m)=f(b_m)<\textbf{0}.
\end{equation}

Then the solution initiated at $b_m$ is decreasing and, again by the monotonicity of the system, for any solution of (\ref{SIT-ode-compact}) initiated in $\mathcal{D}$ we have
\begin{equation}\label{upperbound-SIT}
x_{M_T^*}(z, t)\leq b_{||z||_\infty}.
\end{equation}
The a priori upper bound given in (\ref{upperbound-SIT}) provides for existence of the solution for all $t\geq 0$. Therefore, (\ref{SIT-ode-compact}) defines a dynamical system on $\mathcal{D}$.

\begin{itemize}
    \item[(1)] Suppose that $M_T^*>M_{T_1}$. According to Proposition \ref{SIT-ode-equilibre}, system (\ref{SIT-ode-compact}) has only one equilibrium, namely $\bf{0}$. 
    The global asymptotic stability of $\bf{0}$ is proved as in point 1) of Theorem \ref{Mosquitoes-ode-theorem}.
    \item[(3)] Assume that $0<M_T^*<M_{T_1}$. In this case, the dynamical system (\ref{SIT-ode-compact}) has three equilibria \textbf{0}, $E_1$ and
    $E_2$. Since the eigenvalues, $\xi_1=-(\gamma+\mu)$, $\xi_2=-\mu_M$, $\xi_3=-\mu_F$, of the Jacobian matrix of the SIT model (\ref{SIT-ode}) at \textbf{0} are all negative, then the elimination equilibrium \textbf{0} is locally asymptotically stable. Let us consider the order interval $[\textbf{0}, E_1]$. According to \cite[Theorem 2.2.2]{Smith2008}, the solutions initiated in this interval, excluding the end points, either all converge to \textbf{0} or all converge to $E_1$. Since \textbf{0} is asymptotically stable, this implies that all solutions converge to \textbf{0}. Moreover, straightforward computations lead that the Jacobian matrix, $J_{E_1}$, of the SIT model (\ref{SIT-ode}) at $E_1$ is an irreducible Metzler matrix. Hence, it follows from the theory of nonnegative matrices \cite[Theorems 11 and 17]{Haddad2010}, \cite[Proposition 3.4]{Guiver2016}
     that $J_{E_1}$ has an eigenvector $v$ with positive coordinates and associated eigenvalue $\xi$, which is real and has an algebraic multiplicity equal to one. Since $E_1$ is repelling in $[\textbf{0}, E_1]$, then $\xi\geq0$. In fact, $\xi>0$. Indeed, straightforward computations and, taking into account (\ref{slope-M1}), lead that $$\det(J_{E_1})=(1-r)r\gamma^2\phi A\left(\displaystyle\frac{M_T^*}{(M_1+M_T^*)^2}-\frac{\mu_F\mu_{A,2}}{r\phi}\frac{\mu_M}{(1-r)\gamma^2}\right)>0.$$ Therefore, $\xi>0$ since $\det(J_{E_1})$ is the product of eigenvalues of $J_{E_1}$. Next, we consider the order interval $[E_1, E_2]$. Again following \cite[Theorem 2.2.2]{Smith2008}, we deduce that the solutions initiated in this interval, excluding the end points, all converge to $E_2$ since $E_1$ is repelling in the direction of the positive vector $v$. Now, let $x=x(t)$ be any solution of the SIT model (\ref{SIT-ode}) such that $x(0)\geq E_2$. Denote by $y=y(t)$ the solution of (\ref{SIT-ode}) with initial data $y(0)=b_{\|x(0)\|_\infty}$. It follows from inequality  (\ref{decreasing}) that the function $y$ is decreasing and, therefore, it converges. The limit is necessarily an equilibrium greater or equal to $E_2$. However, there is no other equilibrium greater than $E_2$. Thus, the limit of $y(t)$, as $t$ goes to infinity, is $E_2$. Using that model (\ref{SIT-ode}) is a monotone system, $E_2\leq x(0) \leq y(0)$ implies that $E_2\leq x(t)\leq y(t).$ Hence $\lim\limits_{t\rightarrow+\infty}x(t)=E_2$.
     \par  
     The proof of point (2) is done in a
similar way but by considering $E_1:=E_\dag$ to construct the basin of attraction of the elimination equilibrium and, $E_2:=E_\dag$ to construct the basin of attraction of $E_\dag$.
\end{itemize}
\end{proof}

Fig. \ref{box}, page \pageref{box}, depicts a rough illustration of the bistable case obtained in the last part of Theorem \ref{SIT-ode-theorem}. In
Fig. \ref{box}, the black bullet is the wild equilibrium
$E^*=(A^*,M^*,F^*)'$, the blue bullet is the positive unstable
equilibrium ($E_1$) while the red bullet is the positive stable
equilibrium ($E_2$). The dashed black box is the set $[\textbf{0},
E_1)$ which is contained in the basin of attraction of \textbf{0}
while the solid black box is the set $\{x\in \mathbb{R}: x>E_1\}$
which is contained in the basin of attraction $E_2$.

\begin{figure}[H]
    \centering
      \includegraphics[scale=0.55]{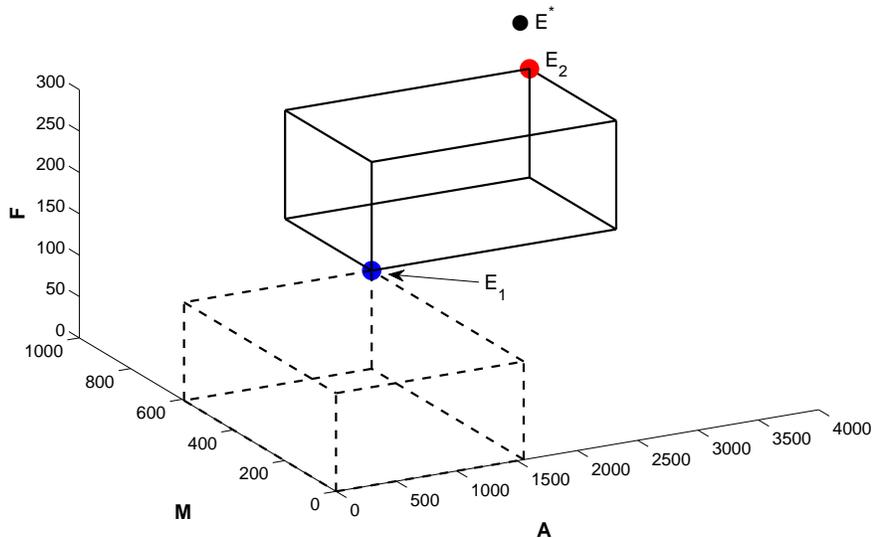}
    \caption{Rough illustration of the bistable case
obtained in the last part of Theorem \ref{SIT-ode-theorem}. The
black bullet is the wild equilibrium $E^*$, the blue bullet is the
positive unstable equilibrium ($E_1$) while the red bullet is the
positive stable equilibrium ($E_2$).}
    \label{box}
\end{figure}

Our aim, with a permanent SIT control is to drive, starting from the equilibrium $E^*$ and using massive releases, the solution of the SIT system inside the set $[\textbf{0}, E_1)$, for a given $M_{T_1}$. Once inside, the LAS property of $\bf{0}$ in $[\textbf{0}, E_1)$ will maintain the solution inside the set $[\textbf{0}, E_1)$. In fact, the solution will slowly, but surely, continue to decay to $\textbf{0}$.

\section{Characterization of the time necessary to reduce the
amount of vector population}\label{caracterisation}

SIT control generally consists of massive releases in the targeted
area in order to reach elimination or to lower the population under
a certain threshold, to reduce the nuisance (bites) or/and the
epidemiological risk. However, according to our theoretical results,
once SIT is stopped, the system will recover. Thus, SIT always needs
to be maintained. However, from a practical point of view, massive
releases can only occur for a limited period of time, such that it
should be followed by small releases. The objective of this section
is to study such a strategy: first, massive releases, followed by
small releases. To do that, we define the size of the small
releases, $\underline{M}_T^*$, we want to reach, and thus define the
subdomain (which belong to the basin of attraction of $\textbf{0}$)
we need to reach in order to start the small releases, such that the
wild population stays inside the targeted subdomain and slowly but
surely even converges to $\textbf{0}$. Our results are based on the
fact that our system is monotone and bi-stable, such that we are
able to build part of the basin of attraction of \textbf{0}, defined
by $[\textbf{0},E_1)$ (see Theorem \ref{SIT-ode-theorem}). Then, we provide lower and upper bounds for the time, $\tau(M_T^*)$, needed to reach the subdomain $[\textbf{0},E_1)$.

\par

As previously stated, the aim of this section is therefore to
estimate the minimal time $\tau(M_T^*)$ necessary for solution of
system (\ref{SIT-ode}) to be in the box $[\textbf{0}, E_1)$ which
ensures the elimination of vectors in the long term dynamic. Here
$E_1$ is the unstable positive equilibrium of system (\ref{SIT-ode})
defined in Proposition \ref{SIT-ode-equilibre} when the size of the release
$M_T^*$ is such that $0<M_T^*<M_{T_1}$. Therefore in the sequel we assume
that $0<\underline{M}_T^*<M_{T_1}$. Let us consider
$\underline{E}_1=(\underline{A}_1,\underline{M}_1,\underline{F}_1)'$
the unstable positive equilibrium of system (\ref{SIT-ode}) that
corresponds to $\underline{M}_T^*$. Let $\varepsilon>0$, we define
$\underline{Y}=(\underline{A},\underline{M},\underline{F})'$ by
\begin{equation}\label{defY}
\left\{
\begin{array}{ccl}
  \underline{A} & = & \underline{A}_1-\varepsilon, \\
  \underline{M} & = & \underline{M}_1-\varepsilon, \\
  \underline{F} & = & \underline{F}_1-\varepsilon.
\end{array}\right.
\end{equation}

Recall that, the wild vectors equilibrium is
$E^*=(A^*,F^*,M^*)'>\underline{E}_1$ with
\begin{equation}
 \left\{%
\begin{array}{rcl}
A^*&=& \displaystyle\frac{(\gamma+\mu_{A,1})}{\mu_{A,2}}(R-1),\\
 M^*&=& \displaystyle\frac{(1-r)\gamma A^*}{\mu_M},\\
 F^*&=& \displaystyle\frac{r\gamma A^*}{\mu_F}.\\
\end{array}
\right.
\end{equation}
 We denote, for $t\geq t_0$, $a,b\in \mathbf{R}$, $X_{t_0}(t, a, b)$ the solution of system (\ref{SIT-ode})
 with $M_T=b$ such that $X_{t_0}(t_0,a,b)=a$. The following result holds true.

\begin{theorem}[Minimal entry time]\label{lemme-min-time-maths}Assume that $M_T^*>M_{T_1}$ with a targeted release $\underline{M}_T^*$, such that $0<\underline{M}_T^*<M_{T_1}$. Consider
${\bf{0}}<\underline{Y}=(\underline{A},\underline{M},\underline{F})'<\underline{E}_1$,
defined by (\ref{defY}). Then, there exists $\tau>0$ such that:
\begin{itemize}
    \item[(i)]$X_0(\tau, E^*, M_T^*)=\underline{Y}$.
    \item[(ii)] For all $t\geq \tau$, $X_\tau(t, \underline{Y}, \underline{M}_T^*)<\underline{E}_1$ and
    $\lim\limits_{t\rightarrow+\infty}X_\tau(t, \underline{Y}, \underline{M}_T^*)=\bf{0}$.
\end{itemize}
\end{theorem}

\begin{proof}When $M_T=M_T^*>M_{T_1}$, it follows from Theorem
\ref{SIT-ode-theorem} that $\textbf{0}$ is globally asymptotically
stable for system (\ref{SIT-ode}) in $\mathbb{R}^3_+$. That is, for
all $\alpha>0$ there exist $t_\alpha>0$ such that for all $t\geq
t_\alpha$ $\|X_0(t, E^*, M_T^*)\|_{\mathbb{R}^3}\leq\alpha$. In
particular, for
$\alpha=\displaystyle\frac{1}{10^n}\|\underline{Y}\|_{\mathbb{R}^3}$,
with $n\in \mathbb{N}^*$ sufficiently large, such $t_\alpha$ exists.
Since ${\bf{0}}<\underline{Y}<E^*$ and by the continuity of
$X_0(t, E^*, M_T^*)$ we deduce that there exist a finite sequence
$(\tau_i)_{i=1,...,p}$ such that $0<\tau_1<\tau_2<...<\tau_p$ and
$X_0(\tau_i, E^*, M_T^*)=\underline{Y}$. We therefore set
$$\tau:=\min\limits_{i=1,...,p}\tau_i$$ and  part $(i)$ of Theorem
\ref{lemme-min-time-maths} holds.
\par

When $M_T=\underline{M}_T^*\in(0,M_{T_1})$, part $(ii)$ follows from the fact that $\bf{0}$ is LAS and the set $[{\bf{0}},\underline{E}_1)$ is contained in
its basin of attraction (Theorem \ref{SIT-ode-theorem})
\end{proof}

\begin{remark}\label{remark-offspring}
It is important to observe that if the massive release is stopped
before the prescribed period of time, $\tau$, obtained in Theorem
\ref{lemme-min-time-maths}, system (\ref{SIT-ode}) will converge
towards the positive stable equilibrium.
\end{remark}

Under certain conditions we can derive an analytic approximation for
the minimal time, $\tau$, defined in Theorem \ref{lemme-min-time-maths}.
We deal with that issue in the sequel. In this section, we assume that
\begin{equation}\label{tech-assump}
    \mu_F<\min\{\mu_M,\gamma+\mu_{A,1}\}.
\end{equation}
Assumption (\ref{tech-assump}) is also supported by parameter values considered, for the case of \textit{Aedes spp.}, in \cite{Anguelov2012TIS, Bliman2019, Chitnis2008}.

The following inequalities holds
\begin{equation}\label{SIT-ode-lower-upper-equilibrium}
 \begin{array}{rcccl}
0&\leq& A^*&\leq& \displaystyle\frac{(\gamma+\mu_{A,1})}{\mu_{A,2}}R:=A_e^0,\\
0&\leq& M^*&\leq& \displaystyle\frac{(1-r)\gamma}{\mu_M}\displaystyle\frac{(\gamma+\mu_{A,1})}{\mu_{A,2}}R:=M_e^0,\\
0&\leq& F^*&\leq&\displaystyle\frac{r\gamma}{\mu_F}\displaystyle\frac{(\gamma+\mu_{A,1})}{\mu_{A,2}}R:=F_e^0.\\
\end{array}
\end{equation}
Let us consider the solution $X(t)=(A(t),M(t),F(t))'$ of system
(\ref{SIT-ode}) with initial data $E^*$. In order to estimate the (minimal) time needed to drive the vector population under a given value $\underline{Y}=(\underline{A},\underline{M},\underline{F})'<E^*$, we will look for an analytical upper bound of $X(t)$, $X^{upper}(t)$.

According to system (\ref{SIT-ode}), we have
\begin{equation}\label{SIT-ode_ineq}
 \left\{%
\begin{array}{lcl}
 \displaystyle\frac{dA}{d t} &\leq & \phi F-(\gamma+\mu_{A,1})A,\\
     \displaystyle\frac{dM}{d t} &=& (1-r)\gamma A-\mu_MM,\\
    \displaystyle\frac{dF}{d t} &=& \displaystyle\frac{ M}{M+M_T}r\gamma A-\mu_FF,\\
\end{array}
\right.
\end{equation}
that is $$\displaystyle\frac{dX}{dt}\leq Z X$$ where
$$Z=\left(
      \begin{array}{ccc}
        -(\gamma+\mu_{A,1}) & 0 & \phi \\
        (1-r)\gamma & -\mu_M & 0 \\
        r\gamma\epsilon(M_T^*) & 0 & -\mu_F \\
      \end{array}
    \right)
$$ and $\epsilon(M_T^*)=M^*/(M^*+M_T^*)<1$.
Let us set $X_e(t)=(A_e(t), M_e(t), F_e(t))'$, the solution of
\begin{equation}\label{SIT-ODE-Comparison-lineaire}
    \displaystyle\frac{dX_e}{dt}= Z X_e.
\end{equation}

Before going further, let us give the following result that is
deduced from Proposition 1.4 and Corollary 1.6 in \cite{Kirkilionis2004} thanks to the fact that systems (\ref{SIT-ode}) and (\ref{SIT-ODE-Comparison-lineaire}) are cooperative systems.

\begin{lemma}\label{comparaison-monotonie}
Solution of systems (\ref{SIT-ode}) and
(\ref{SIT-ODE-Comparison-lineaire}) with initial data such that
$$(A^0, M^0, F^0)'\leq (A^0_e,M^0_e,F^0_e)':=X_e^0$$ satisfy $$\forall
t\geq0,\quad X(t)\leq X_e(t).$$
\end{lemma}

In the sequel, we follow the idea of \cite{Strugarek2019} in our computations. The sub-matrix $Z_0$ of
$Z$ that reads as

$$Z_0=\left(
      \begin{array}{cc}
        -(\gamma+\mu_{A,1}) & \phi \\
        r\gamma\epsilon(M_T^*) & -\mu_F \\
      \end{array}
    \right)
$$ has negative trace. Moreover, $Z_0$ has a positive determinant if
and only if $1/R>\epsilon(M_T^*)$. Therefore, if
$\epsilon(M_T^*)R<1$ then \textbf{0} is globally asymptotically
stable for system (\ref{SIT-ODE-Comparison-lineaire}). In this case,
its eigenvalues are real, negative and equal to $\kappa_\pm$
($\kappa_-<\kappa_+$) associated respectively with eigenvectors
$\left(
                     \begin{array}{c}
                       1 \\
                       x_\pm \\
                     \end{array}
                   \right)
$ where, with assumption (\ref{tech-assump}) $x_-<0<x_+$ and
$$
\begin{array}{ccl}
  \kappa_\pm & = & \displaystyle\frac{-(\gamma+\mu_{A,1}+\mu_F)\pm\sqrt{(\gamma+\mu_{A,1}-\mu_F)^2+4\phi r\gamma\epsilon(M_T^*)}}{2}, \\
  x_\pm & = &
    \displaystyle\frac{\gamma+\mu_{A,1}-\mu_F\pm\sqrt{(\gamma+\mu_{A,1}-\mu_F)^2+4\phi r\gamma\epsilon(M_T^*)}}{2\phi}.
\end{array}
$$

Hence for real numbers $(a^0_\pm, b^0_{\pm})'\in \mathbb{R}^4$, we
have
$$\left(
                     \begin{array}{c}
                       A_e(t) \\
                       M_e(t)\\
                       F_e(t) \\
                     \end{array}
                   \right)=\left(
                     \begin{array}{l}
                       a^0_+e^{\kappa_+t}+a^0_-e^{\kappa_-t} \\
                       e^{-\mu_Mt}M_e^0+(1-r)\gamma\displaystyle\int_0^te^{-\mu_M(t-s)}(a^0_+e^{\kappa_+s}+a^0_-e^{\kappa_-s})ds\\
                       b^0_+e^{\kappa_+t}+b^0_-e^{\kappa_-t} \\
                     \end{array}\right)
                   $$
where $a^0_\pm, b^0_{\pm}$ are computed by using the overestimation
$(A_e^0, F_e^0)'$ in (\ref{SIT-ode-lower-upper-equilibrium}) as
initial condition. In details, we found
$$\left\{
\begin{array}{rcl}
   a^0_+& = & \displaystyle\frac{x_-A_e^0-F_e^0}{x_--x_+}, \qquad a_-^0=\displaystyle\frac{-x_+A_e^0+F_e^0}{x_--x_+}, \\
  b^0_+& = & \displaystyle\frac{x_+x_-A_e^0-x_+F_e^0}{x_--x_+}, \qquad  b_-^0=\displaystyle\frac{-x_+x_-A_e^0+x_-F_e^0}{x_--x_+}.
\end{array}\right.
$$
Note that
$$x_--x_+=-\displaystyle\frac{\sqrt{(\gamma+\mu_{A,1}-\mu_F)^2+4\phi r\gamma\epsilon(M_T^*)}}{\phi}<0,$$
$a_+^0>0$, $b_+^0>0$, $a_-^0<0$ and $b_-^0=x_-a_-^0>0$. Indeed, for
$\Delta=(\gamma+\mu_{A,1}-\mu_F)^2+4\phi r\gamma\epsilon(M_T^*)$ we
have
$$
\begin{array}{ccl}
  a_-^0<0 & \Leftrightarrow& x_+A_e^0<F_e^0 \\
   &\Leftrightarrow  & \displaystyle\frac{((\gamma+\mu_{A,1})-\mu_F+\sqrt{\Delta})}{2\phi}\displaystyle\frac{(\gamma+\mu_{A,1})R}{\mu_{A,2}}<\displaystyle\frac{r\gamma(\gamma+\mu_{A,1})R}{\mu_F\mu_{A,2}} \\
   & \Leftrightarrow & (\gamma+\mu_{A,1})-\mu_F+\sqrt{\Delta}<\displaystyle\frac{2\phi r\gamma}{\mu_{F}} \\
   & \Leftrightarrow & \sqrt{\Delta}<(\gamma+\mu_{A,1})(2R-1)+\mu_F \\
   & \Leftrightarrow &
   r\gamma\phi\epsilon(M_T^*)<(\gamma+\mu_{A,1})^2R(R-1)+r\gamma\phi\\
   &\Leftrightarrow& r\gamma\phi(\epsilon(M_T^*)-1)<0<(\gamma+\mu_{A,1})^2R(R-1).\\
\end{array}
$$

In addition, by using assumption (\ref{tech-assump}) we also have
$$\kappa_++\mu_M=\displaystyle\frac{2(\mu_M-\mu_F)-(\gamma+\mu_{A,1}-\mu_F)+\sqrt{\Delta}}{2}>0.$$
Moreover, assuming $\kappa_-\neq-\mu_M$ (which most holds generally)
leads that

$$\begin{array}{ccl}
    M_e(t)&=&e^{-\mu_Mt}M_e^0+(1-r)\gamma\left(a_+^0\displaystyle\frac{e^{\kappa_+t}-e^{-\mu_Mt}}{\mu_M+\kappa_+}+a_-^0\displaystyle\frac{e^{\kappa_-t}-e^{-\mu_Mt}}{\mu_M+\kappa_-}\right) \\
     & = & \left(M_e^0-\displaystyle\frac{(1-r)\gamma a_+^0}{\mu_M+\kappa_+}-\displaystyle\frac{(1-r)\gamma
     a_-^0}{\mu_M+\kappa_-}\right)e^{-\mu_Mt}+\displaystyle\frac{(1-r)\gamma a_+^0}{\mu_M+\kappa_+}e^{\kappa_+t}+\displaystyle\frac{(1-r)\gamma
     a_-^0}{\mu_M+\kappa_-}e^{\kappa_-t}.
  \end{array}
$$
Before going further, recall that
$$0<\underline{Y}<E^*<X_e^0.$$

Since $a_-^0<0$, $A_e(t)\leq \underline{A}$ if
$a_+^0e^{\kappa_+t}\leq \underline{A}.$ That is if
\begin{equation}\label{min-t-A}
    t\geq
    t^A_{min}:=\displaystyle\frac{1}{\kappa_+}\log\left(\displaystyle\frac{\underline{A}}{a^0_+}\right).
\end{equation}

By using the fact that $b^0_++b^0_-=F_e^0$, we deduce that
$F_e(t)\leq \underline{F}$ if $F_e^0e^{\kappa_+t}\leq
\underline{F}.$ That is if
\begin{equation}\label{min-t-F}
    t\geq
    t^F_{min}:=\displaystyle\frac{1}{\kappa_+}\log\left(\displaystyle\frac{\underline{F}}{F^0_e}\right).
\end{equation}

We proved that $\kappa_++\mu_M>0$ but we need to discuss the two
cases $\kappa_-+\mu_M>0$ and $\kappa_-+\mu_M<0$.
\par
 In the case that
$\kappa_-+\mu_M>0$, with $a_-^0<0$ we have

$$M_e(t)\leq \left(M_e^0-\displaystyle\frac{(1-r)\gamma
     a_-^0}{\mu_M+\kappa_-}\right)e^{-\mu_Mt}+\displaystyle\frac{(1-r)\gamma a_+^0}{\mu_M+\kappa_+}e^{\kappa_+t}.$$
Since $\kappa_+>\mu_M$, we obtain
$$M_e(t)\leq \left(M_e^0-\displaystyle\frac{(1-r)\gamma
     a_-^0}{\mu_M+\kappa_-}+\displaystyle\frac{(1-r)\gamma a_+^0}{\mu_M+\kappa_+}\right)e^{\kappa_+t}:=\lambda_-e^{\kappa_+t}$$
     where $\lambda_-=M_e^0-\displaystyle\frac{(1-r)\gamma
     a_-^0}{\mu_M+\kappa_-}+\displaystyle\frac{(1-r)\gamma
     a_+^0}{\mu_M+\kappa_+}>0.$ Therefore, $M_e(t)\leq \underline{M}$ if $\lambda_-e^{\kappa_+t}\leq
\underline{M}.$ That is if
\begin{equation}\label{min-t-M+}
    t\geq
    t^M_{min}:=\displaystyle\frac{1}{\kappa_+}\log\left(\displaystyle\frac{\underline{M}}{\lambda_-}\right).
\end{equation}

In the case that $\kappa_-+\mu_M<0$, with $a_-^0<0$ we have

$$M_e(t)\leq M_e^0e^{-\mu_Mt}+\displaystyle\frac{(1-r)\gamma a_+^0}{\mu_M+\kappa_+}e^{\kappa_+t}+\displaystyle\frac{(1-r)\gamma a_-^0}{\mu_M+\kappa_-}e^{\kappa_-t}.$$
Since $\kappa_+>\mu_M$ and $\kappa_+>\kappa_-$, we obtain

$$M_e(t)\leq \left(M_e^0+\displaystyle\frac{(1-r)\gamma
     a_-^0}{\mu_M+\kappa_-}+\displaystyle\frac{(1-r)\gamma a_+^0}{\mu_M+\kappa_+}\right)e^{\kappa_+t}:=\lambda_+e^{\kappa_+t}$$

     where $\lambda_+=M_e^0+\displaystyle\frac{(1-r)\gamma
     a_-^0}{\mu_M+\kappa_-}+\displaystyle\frac{(1-r)\gamma
     a_+^0}{\mu_M+\kappa_+}>0.$ Therefore, $M_e(t)\leq \underline{M}$ if $\lambda_+e^{\kappa_+t}\leq
\underline{M}.$ That is if
\begin{equation}\label{min-t-M-}
    t\geq
    t^M_{min}:=\displaystyle\frac{1}{\kappa_+}\log\left(\displaystyle\frac{\underline{M}}{\lambda_+}\right).
\end{equation}

 Hence, we have proved the following result.

\begin{proposition}\label{proposition-tau-upper}Let $(A(t), M(t), F(t))'$ be a solution of system
(\ref{SIT-ode}) initiated at the wild equilibrium
$E^*=(A^*,M^*,F^*)'$. Assume that  $\epsilon(M_T^*)R<1$ where
$\epsilon(M_T^*)=M^*/(M^*+M_T^*)$. The necessary time $\tau(M_T^*)$
to lower the vector population from $E^*$ to
$\underline{Y}=(\underline{A}, \underline{M},\underline{F})'$ defined
in (\ref{defY}) with $\underline{A}<A^*$, $\underline{M}<M^*$ and
$\underline{F}<F^*$ is such that
$$\tau(M_T^*)\geq\max(t^A_{min}, t^M_{min}, t^F_{min})$$ where $t^A_{min}$
is given by (\ref{min-t-A}), $t^F_{min}$ is given by (\ref{min-t-F})
and $t^M_{min}$ is given by (\ref{min-t-M+}) or (\ref{min-t-M-}).
\end{proposition}

\section{SIT with periodic impulsive releases}\label{SIT-Pulse}
Continuous releases, while mathematically very convenient, are not realistic. In general releases are periodic and instantaneous. That is why, we consider the following SIT model with periodic impulsive releases
\begin{equation}\label{SIT-ode-pulse}
 \left\{%
\begin{array}{rcl}
 \displaystyle\frac{dA}{d t} &=& \phi F-(\gamma+\mu_{A,1}+\mu_{A,2}A)A,\\
     \displaystyle\frac{dM}{d t} &=& (1-r)\gamma A-\mu_MM,\\
    \displaystyle\frac{dF}{d t} &=& \displaystyle\frac{ M}{M+M_T}r\gamma A-\mu_FF,\\
    \displaystyle\frac{dM_T}{d t} &=& -\mu_TM_T,\\
    M_T(n\tau^+) &=& M_T(n\tau)+\tau\Lambda,\quad n=1,2,...
\end{array}
\right.
\end{equation}
where $\tau$ (in unit of time) is the pulse release period. The right-hand side of system (\ref{SIT-ode-pulse}) is locally Lipschitz continuous on $\mathbb{R}^4$. Thus, using a classic existence theorem (Theorem 1.1, p. 3 in \cite{BP1993}), there exists $T^*>0$ and a unique solution defined from $(0,T^*)\rightarrow \mathbb{R}^4$. Then, using standard arguments, we show that the positive orthant $\mathbb{R}^4$ is an invariant region for system (\ref{SIT-ode-pulse}).

From the last two equations of system (\ref{SIT-ode-pulse}), we
deduce that, as $t\rightarrow+\infty$, $M_T$ converges toward the
periodic solution

\begin{equation}\label{MT-pulse}
  M_T^{per}(t)= \displaystyle\frac{ \tau\Lambda}{1-e^{-\mu_T\tau}}e^{-\mu_T(t-\lfloor
  t/\tau\rfloor\tau)}.
\end{equation}

Thus, solutions of system (\ref{SIT-ode-pulse}) converges, in the sense of $L^{\infty}(0,+\infty)$ norm, to solutions of the following system
\begin{equation}\label{SIT-ode-pulse-2}
 \left\{%
\begin{array}{rcl}
 \displaystyle\frac{dA}{d t} &=& \phi F-(\gamma+\mu_{A,1}+\mu_{A,2}A)A,\\
     \displaystyle\frac{dM}{d t} &=& (1-r)\gamma A-\mu_MM,\\
    \displaystyle\frac{dF}{d t} &=& \displaystyle\frac{ M}{M+M_T^{per}(t)}r\gamma A-\mu_FF.\\
\end{array}
\right.
\end{equation}
System (\ref{SIT-ode-pulse-2}) is a periodic monotone dynamical system that admits one solution $X$. Substituting
\begin{equation}\label{MT-pulse-lower}
 \underline{M}_T:=\min\limits_{t\in[0,\tau]}M_T^{per}(t)=\displaystyle\frac{
\tau\Lambda}{1-e^{-\mu_T\tau}}e^{-\mu_T\tau},
\end{equation}
in system (\ref{SIT-ode-pulse-2}) leads to the following constant SIT model
\begin{equation}\label{SIT-ode-pulse-cst}
 \left\{%
\begin{array}{rcl}
 \displaystyle\frac{dA}{d t} &=& \phi F-(\gamma+\mu_{A,1}+\mu_{A,2}A)A,\\
     \displaystyle\frac{dM}{d t} &=& (1-r)\gamma A-\mu_MM,\\
    \displaystyle\frac{dF}{d t} &=& \displaystyle\frac{ M}{M+\underline{M}_T}r\gamma A-\mu_FF.\\
\end{array}
\right.
\end{equation}
whose solution $X_M$ is such that $X_M \geq X$ for all time $t > 0$, using a comparison principle. Thus applying to system (\ref{SIT-ode-pulse-cst}) the results obtained in Theorems \ref{SIT-ode-theorem} and  \ref{lemme-min-time-maths}, we obtain conditions on the size and the periodicity of the releases to get GAS or LAS of $\bf{0}$.
Using $M_{T_1}$ defined in (\ref{MT1}), we set
\begin{equation}
    M_{T_1}^{per}=M_{T_1}\left(e^{\mu_T\tau}-1\right).
\end{equation}
$M_{T_1}^{per}$ is not the best release value for the periodic case. Most probably the best release value should depend on $\frac{1}{\tau}\int_{0}^{\tau} \frac{1}{M_T^{per}(t)}dt$, like in \cite{Bliman2019}.
Then, following Theorem \ref{SIT-ode-theorem},  we deduce
\begin{proposition}\label{SIT-Pulse-proposition}
For $\tau$ and $\Lambda$ given, and 
\begin{itemize} 
\item[(i)] Assuming \begin{equation}\label{Seuil-pulse-release}
 \tau\Lambda > M_{T_1}^{per},
\end{equation}
then $\bf{0}$ is globally asymptotically stable in (\ref{SIT-ode-pulse-2}).
    \item[(ii)]Assuming \begin{equation}
 \tau\Lambda = M_{T_1}^{per},
\end{equation}
then $\bf{0}$ is locally asymptotically stable in
(\ref{SIT-ode-pulse-2}), and $[{\bf{0}}, E_\dagger( \underline{M}_T))$ lies in
its basin of attraction. 
    \item[(iii)]Assuming \begin{equation}
 0<\tau\Lambda < M_{T_1}^{per},
\end{equation}
then \textbf{0} is locally asymptotically stable in
(\ref{SIT-ode-pulse-2}), and $[{\bf{0}}, E_1(\underline{M}_T))$ lies in its
basin of attraction. 
\end{itemize}
\end{proposition}

Using Theorem \ref{lemme-min-time-maths}, we deduce
\begin{theorem}\label{lemme-min-time-maths-pulse}
Let 
${\bf{0}}<\underline{Y}=(\underline{A},\underline{M},\underline{F})'<{E}_1(\underline{M}_{T}^*)$, as defined in (\ref{defY}), for a given target release amount, $\underline{M}_{T}^*<M_{T_1}^{per}$. The following results hold
\begin{itemize}
\item First, assuming massive releases, with $\tau\Lambda^* >M_{T_1}^{per}$, then $M_{T}^{per}$ converges from $E^*$ to \underline{Y} in a finite time $t^*>0$. 
\item Second, assuming small releases, with $\tau\Lambda^* =\underline{M}_{T}^*$, then, for $t>t^*$, $M^{per}_T(t)<\underline{Y}$ and $\lim\limits_{t\rightarrow+\infty}M^{per}_T(t)=\bf{0}$.

\end{itemize}
\end{theorem}

Theorems \ref{lemme-min-time-maths-pulse} and \ref{lemme-min-time-maths} give us a strategy to drive, in a finite time, and keep the wild vector population under a given threshold value $\underline{Y}$, for a targeted amount of sterile male releases, namely $\underline{M}_{T}^*$: first, massive releases for several weeks, and then small releases according to $\underline{M}_{T}^*$. 
They are illustrated in the forthcoming section, both for constant and periodic impulsive releases.

\section{Numerical simulations}\label{Numerical-simulations}
In this part, we consider a specific application of SIT against mosquito, like anopheles or aedes spp. We consider mosquito parameter values given in Table
\ref{Table-valeurs-parametre}.
\begin{table}[H]
  \centering
  \begin{tabular}{|c|c|c|c|c|c|c|c|}
  \hline
  Symbol & $\phi$ & $\mu_{A,1}$ & $\mu_{A,2}$ & $r$ & $\mu_F$ & $\mu_M$ & $\mu_T$ \\
  \hline
  Value & 10 & 0.05 & 2$\times10^{-4}$ & 0.49 & 0.1 & 0.14 & 0.14 \\
  \hline
\end{tabular}
\caption{Entomological parameter values
}\label{Table-valeurs-parametre}
\end{table}
In Table
\ref{Table-Wild-equilibrium} we provide several computations, related to the maturation rate, $\gamma$. We derive the wild (positive) equilibrium
$E^*=(A^*,M^*,F^*)'$ according (\ref{Wild-equilibria-definition}). These wild equilibria will be used as the initial data for forthcoming simulations. 
In addition, we also display in Table \ref{Table-Wild-equilibrium}  the thresholds related to the global asymptotic stability of \textbf{0} with constant release ($M_{T_1}$) and periodic pulse release ($M_{T_1}^{per}$).

\begin{table}[H]
  \centering
  \begin{tabular}{|c|c|c|c|c|}
\hline
  $\gamma$ & 0.04   & 0.06  & 0.08   & 0.1 \\
  \hline\hline
$\mathcal{R}$        & 21.78   & 26.73     & 30.15    & 32.67 \\
    \hline\hline
$A^*$      & 9350   & 14150  & 18950 & 23750 \\
   \hline
$M^*$      & 1335  & 3031   & 5412 & 8479 \\
   \hline
$F^*$      & 1834 & 4160 & 7428  & 11637 \\
\hline\hline
  $M_{T_1}$  & 863.9   & 2048   & 3745  & 5954 \\
   \hline
  $M_{T_1}^{per}$  & 1484.5   & 3519.8   & 6434.3  & 10230 \\
    \hline
\end{tabular}
\caption{Wild equilibrium $E^*=(A^*,M^*,F^*)'$ and Threshold values for $\gamma$ with periodic treatment $\tau=7$ days.
   }\label{Table-Wild-equilibrium}
\end{table}

We now compute the minimal time $\tau(M_T^*)$ necessary for the solution
of system (\ref{SIT-ode}) initiated at the wild equilibrium $E^*$ to
enter in the box $[{\bf{0}}, E_1)$ (see Figure
\ref{ODE-TIS-3D-cont}-(b) and figure \ref{ODE-TIS-3D-pulse}-(b)).
Existence of such time was proved in Theorem 
\ref{lemme-min-time-maths}.
 \par

In Table \ref{Table-Positive-equilibrium}, for a given amount of sterile males to release, $\underline{M}_T^*$, we provide the values of the positive unstable equilibrium $E_1=(A_1,M_1,F_1)'$. This is needed to define $\underline{Y}=(A_1-\varepsilon,
F_1-\varepsilon,M_1-\varepsilon)'$, for a given $\varepsilon>0$, and thus to estimate the minimal time. In the forthcoming simulations, we set $\varepsilon=0.1$. 

\begin{table}[H]
  \centering
  \begin{tabular}{|c|c|c|c|}
\hline
  \backslashbox{$\gamma$}{$\underline{M}_T^*$} & 100 & 500 & 800\\
  \hline
0.04   & (36.59,5.2,0.36)' & (283.11,40.43,4.15)'  & (878.68,125.48,23.35)'  \\
  \hline
0.06   & (18.79,4.03,0.21)' & (109.67,23.49,1.45)'  & (201.11,43.1,3.02)'  \\
   \hline

0.08   & (12.24,3.5,0.16)' & (66.42,18.97,0.95)' & (113.54,32.4,1.7)' \\
   \hline
0.1   & (8.95,3.2,0.14)' & (47.1,16.8,0.75)' & (78.4,27.9,1.3)' \\
   \hline
\end{tabular}
\caption{Values of the  positive (unstable) equilibrium
$E_1=(A_1,M_1,F_1)'$ that corresponds to
  the targeted release $\underline{M}_T^*$ and $\gamma$.}\label{Table-Positive-equilibrium}
\end{table}

The next simulations are done using standard odes routines, implemented in Matlab.
\subsection{Minimal time in the case of continuous and constant releases}
We consider massive constant releases, $M_T^*=k\times M_{T_1}$ (see Table \ref{Table-Wild-equilibrium} for $M_{T_1}$). The minimal entry time for different values of $k$, $\gamma$, and $\underline{M}_T^*$ are summarized in Tables \ref{Table-time-01} and \ref{Table-time-1}. 
 
\begin{table}[H]
  \centering
  \begin{tabular}{|c|c|c|c|c|c|c|c|c|c|c|c|}
\cline{1-4}\cline{6-8}\cline{10-12}
\multicolumn{4}{|c|}{$k=1.001$}&&\multicolumn{3}{|c|}{$k=1.01$}&&\multicolumn{3}{|c|}{$k=1.1$}\\
  \cline{1-4}\cline{6-8}\cline{10-12}
  \backslashbox{$\gamma$}{$\underline{M}_T^*$} & 100 & 500 & 800 && 100 & 500 & 800&& 100 & 500 & 800\\
  \cline{1-4}\cline{6-8}\cline{10-12}
0.04   & 6959 & 6889  & 6719 && 2159 & 2090 & 1929 && 656 & 592 & 479 \\
   \cline{1-4}\cline{6-8}\cline{10-12}
0.08   & 7151 & 7123 & 7112 && 2224 & 2196 & 2184 && 685 & 658 & 647 \\
   \cline{1-4}\cline{6-8}\cline{10-12}
\end{tabular}
\caption{The case of continuous and constant release. Numerical
  estimates of the minimal times (in days) to reach $\underline{Y}$
  We set $\varepsilon=0.1$ using massive releases, $M_T^*=k\times
  M_{T_1}$.}\label{Table-time-01}
\end{table}
\begin{table}[H]
  \centering
\begin{tabular}{|c|c|c|c|c|c|c|c|c|c|c|c|c|c|c|c|}
\cline{1-4}\cline{6-8}\cline{10-12}\cline{14-16}
\multicolumn{4}{|c|}{$k=1.2$}&&\multicolumn{3}{|c|}{$k=2$}&&\multicolumn{3}{|c|}{$k=5$}&&\multicolumn{3}{|c|}{$k=10$}\\
  \cline{1-4}\cline{6-8}\cline{10-12}\cline{14-16}
  \backslashbox{$\gamma$}{$\underline{M}_T^*$} & 100 & 500 & 800 && 100 & 500 & 800&& 100 & 500 & 800&& 100 & 500 & 800\\
  \cline{1-4}\cline{6-8}\cline{10-12}\cline{14-16}
0.04   & 460 & 399  & 311 && 217 & 169 & 126 && 141 & 103 & 76 && 123 & 88 & 65 \\
\cline{1-4}\cline{6-8}\cline{10-12}\cline{14-16}
0.06   & 476 & 442  & 426 && 232 & 201 & 188 && 155 & 128 & 117 && 137 & 111 & 101 \\
   \cline{1-4}\cline{6-8}\cline{10-12}\cline{14-16}
0.08   & 485 & 458 & 447 && 239 & 214 & 205&& 162 & 139 & 131 && 144 & 122 & 114\\
   \cline{1-4}\cline{6-8}\cline{10-12}\cline{14-16}
0.1   & 489 & 465 & 457 && 244 & 221 & 213&& 167 & 146 & 139 && 149 & 129 & 122\\
   \cline{1-4}\cline{6-8}\cline{10-12}\cline{14-16}
\end{tabular}
\caption{The case of continuous and constant release. Numerical
  estimates of the minimal times (in days) to reach $\underline{Y}$, using massive releases, $M_T^*=k\times
  M_{T_1}$.}
  \label{Table-time-1}
\end{table}

For different values of $\underline{M}_T^*$, an increase in the size of the massive releases implies a decay of the minimal time to enter $[{\bf{0}}, E_1)$. Of course lower is the value of $\underline{M}_T^*$, longer is the duration of the massive releases. However, it is interesting to notice that between $k=5$ (where $M_T^*\in [4320,29770]$) and $k=10$ (where $M_T^*\in [8640,59540]$), the gain of time is  very weak if we take into account the cost and, eventually, a possible limitation in the production capacity of the sterile males. Last but not least, when $\underline{M}_T^*=100$ the impact of $\gamma$ on the minimal time, is limited.

To illustrate the trajectory of the SIT system in the constant release case, we provide in  a 3D-view, the trajectory related to $\gamma=0.04$ and $k=5$ (see Figure \ref{ODE-TIS-3D-cont}).
\begin{figure}[H]
    \centering
    \subfloat[][]{  \includegraphics[scale=0.32]{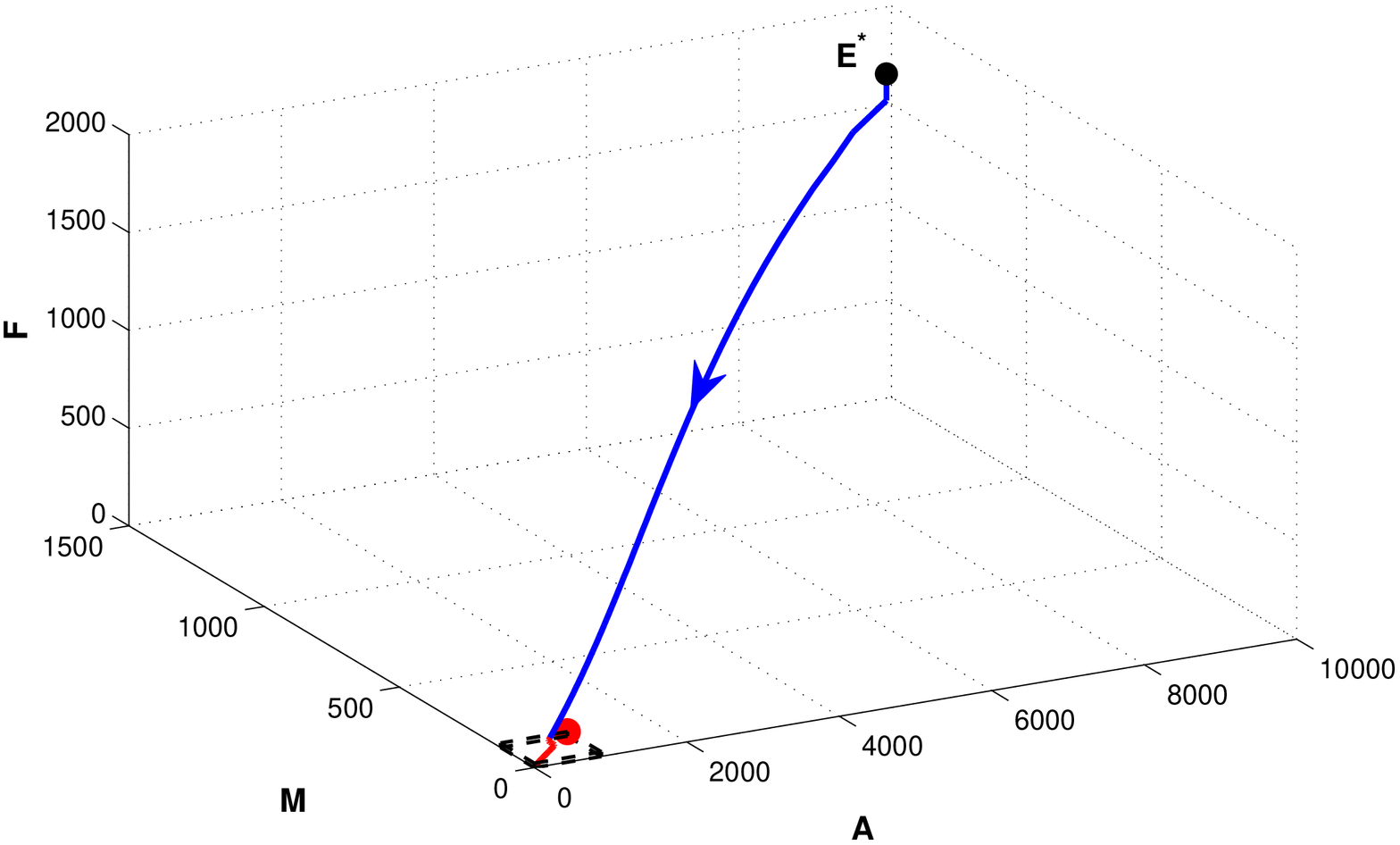}}
    \vspace{0.2cm}
    \subfloat[][]{  \includegraphics[scale=0.32]{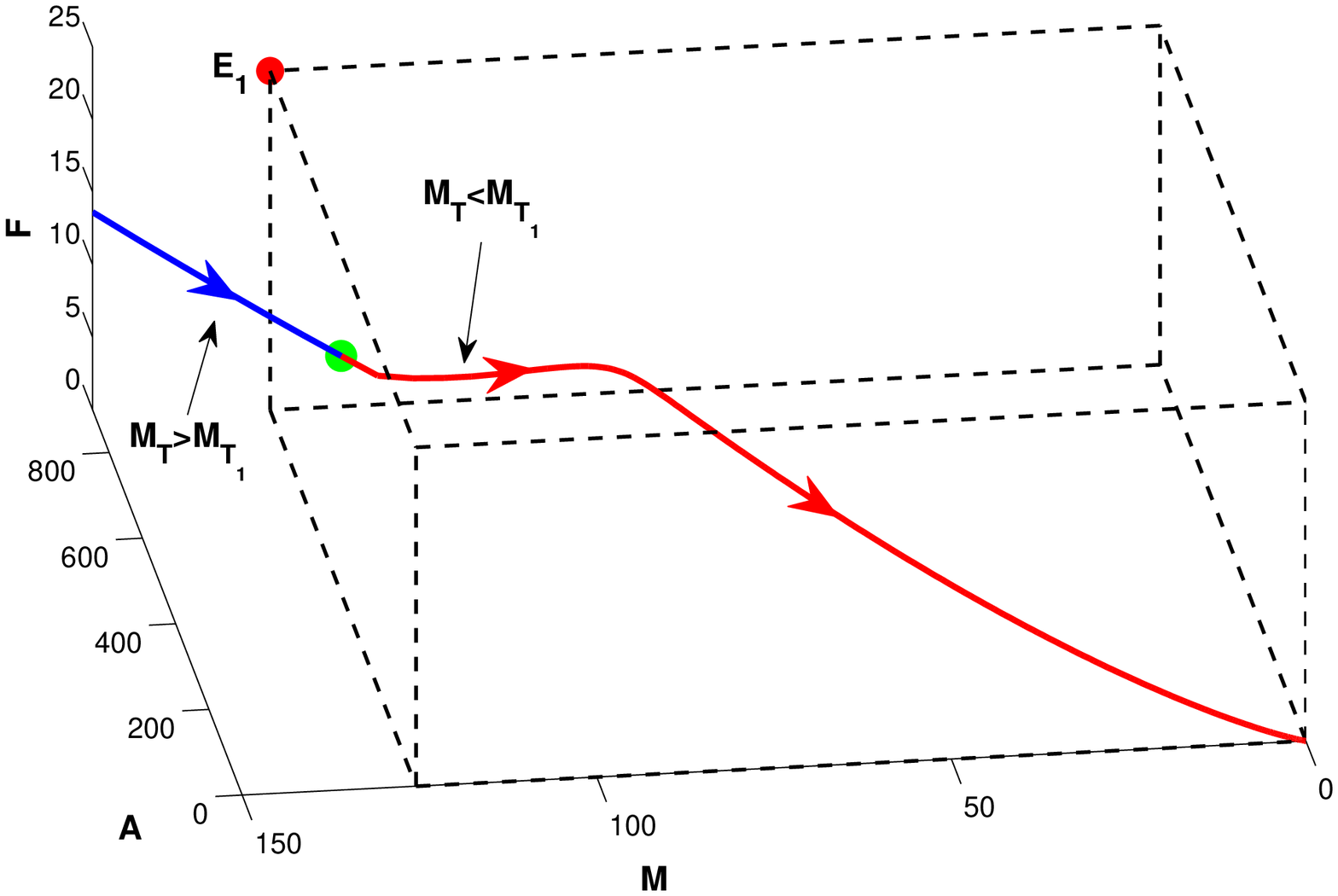}}
    \caption{The case of continuous and constant release. (a): 3D plot of the trajectory of system (\ref{SIT-ode})
    initiated at the wild equilibrium $E^*=(9350,1335,1834)'$ (black
    dot). (b): Zoom in around the box delimited by the positive unstable
    equilibrium $\underline{E}_1=(878.68,125.48,23.35)'$ (red dot).
    The green dot with coordinates (633.2,121.2,10.85)' corresponds
    to the start of the targeted release $\underline{M}_T^*$.
    $\gamma=0.04$, $k=5$, $M_{T_1}=863.9$ and $\underline{M}_T^*=800$.}
    \label{ODE-TIS-3D-cont}
\end{figure}
Note that the red trajectory continues to decay to $\bf{0}$ (because of the LAS of $\bf{0}$), but this is very slow. However, the main objective is achieved: \tcm{to }maintain the wild population below $E_1$.

\subsection{Minimal time in the case of periodic pulse releases}
We consider that releases are done every week, i.e. $\tau=7$. Thus for a given $\tau$, we choose $\Lambda$ such that $\tau \Lambda > M^{per}_{T_1}$. In Table \ref{Table-time-1-pulse}, we provide the results for different values of $k$, $\gamma$, and $\underline{M}_T^*$.
\begin{table}[H]
  \centering
  \begin{tabular}{|c|c|c|c|c|c|c|c|c|c|c|c|c|c|c|c|}
\cline{1-4}\cline{6-8}\cline{10-12}\cline{14-16}
\multicolumn{4}{|c|}{$k=1.2$}&&\multicolumn{3}{|c|}{$k=2$}&&\multicolumn{3}{|c|}{$k=5$}&&\multicolumn{3}{|c|}{$k=10$}\\
  \cline{1-4}\cline{6-8}\cline{10-12}\cline{14-16}
  \backslashbox{$\gamma$}{$\underline{M}_T^*$} & 100 & 500 & 800 && 100 & 500 & 800&& 100 & 500 & 800&& 100 & 500 & 800\\
  \cline{1-4}\cline{6-8}\cline{10-12}\cline{14-16}
0.04   & 213 & 166  & 123 && 166 & 120 & 88 && 127 & 91 & 67 && 117 & 83 & 61 \\
   \cline{1-4}\cline{6-8}\cline{10-12}\cline{14-16}
   0.06   & 228 & 195  & 184 && 175 & 147 & 135 && 140 & 114 & 104 && 130 & 105 & 95 \\
   \cline{1-4}\cline{6-8}\cline{10-12}\cline{14-16}
0.08   & 235 & 210 & 201 && 183 & 159 & 150&& 148 & 125 & 118 && 138 & 116 & 108 \\
   \cline{1-4}\cline{6-8}\cline{10-12}\cline{14-16}
0.1  & 240 & 218 & 210 && 187 & 166 & 159&& 152 & 132 & 125 && 142 & 122 & 115\\
   \cline{1-4}\cline{6-8}\cline{10-12}\cline{14-16}
\end{tabular}
\caption{Periodic impulsive releases are done every 7 days. Numerical
  estimates of the minimal times (in days) to reach $\underline{Y}$, using massive periodic impulsive releases, $M_T^*=
  \Lambda\tau\geq k\times M_{T_1}^{per}$}
  \label{Table-time-1-pulse}
\end{table}

In Figure \ref{ODE-TIS-3D-pulse}, we illustrate the periodic impulsive SIT control for $\gamma=0.04$ and $k=5$. First, with massive periodic releases, followed by small periodic releases. Again, the red trajectory indicates that the system converges (but very slowly) to $\bf{0}$.

\begin{figure}[H]
    \centering
    \subfloat[][]{  \includegraphics[scale=0.32]{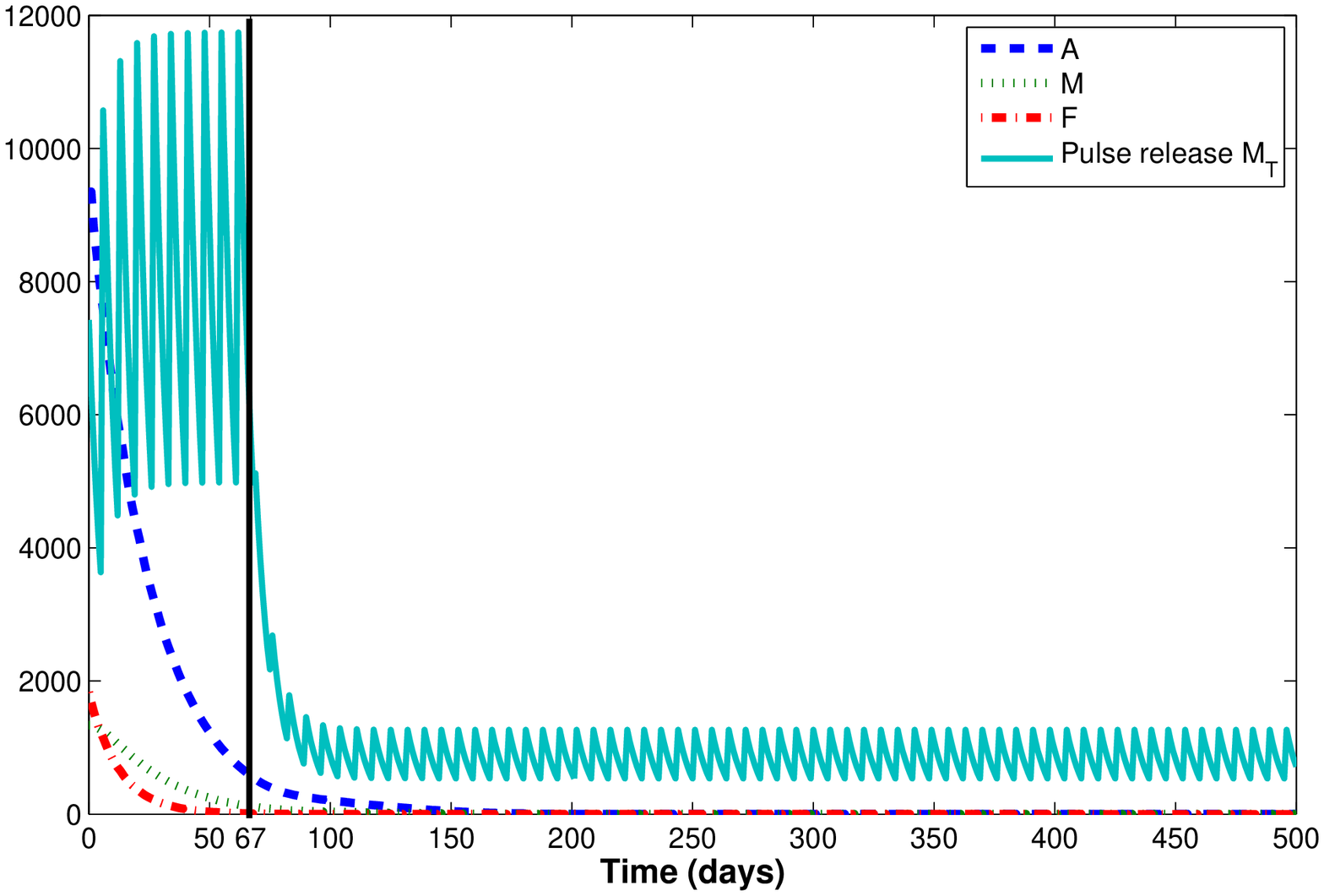}}
    \vspace{0.2cm}
    \subfloat[][]{  \includegraphics[scale=0.32]{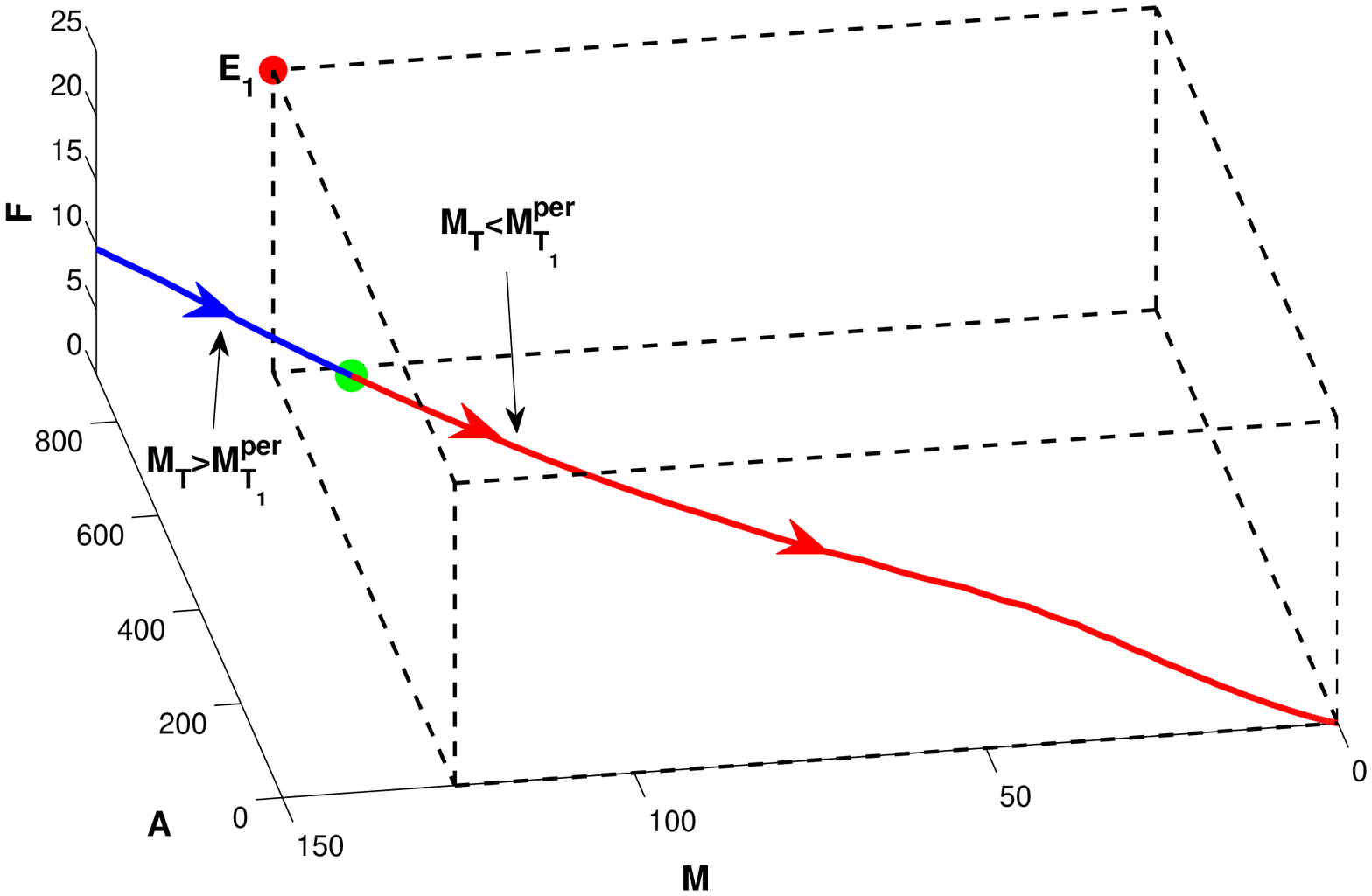}}
    \caption{The case of periodic pulse release. Releases are done every 7 days. (a): Time-serie of the trajectory of system (\ref{SIT-ode})
    initiated at the wild equilibrium $E^*$.
    The solid vertical black line denotes the shift from massive release to targeted release (b): Zoom in around the box delimited by the positive unstable
    equilibrium $\underline{E}_1=(878.68,125.48,23.35)'$ (red dot). The green dot with coordinates ( 604.57,122.4, 9.57)' corresponds
    to the start of the targeted release $\underline{M}_T^*$.
    $\gamma=0.04$, $k=5$, $M_{T_1}^{per}=1484.5$ and $\underline{M}_T^*=800$.}
    \label{ODE-TIS-3D-pulse}
\end{figure}

Comparing the results between Table \ref{Table-time-1-pulse} and Table \ref{Table-time-1}, clearly shows similar results. In fact, the periodic impulsive case is strongly related to the constant release case, thanks to the fact that $<M_{T}^{per}>=\frac{1}{\tau}\int_{t_n}^{t_n+\tau} M_T^{per}(t)dt=\frac{\Lambda}{\mu_T}=M_T^*$. Thus, releasing $\tau \Lambda$ sterile individuals every $\tau$ days is equivalent of releasing a constant amount, $M_T$, of sterile males over the same period. Thus, since $M_{T_1}^{per}=k\times (e^{\mu_T \tau}-1)M_{T_1}$, as long as $k\times (e^{\mu_T \tau}-1)>1$, choosing $\Lambda$ such that $\tau \Lambda >k(e^{\mu_T \tau}-1)M_{T_1}$, is equivalent of choosing $M_T^*=k\times M_{T_1}$. That is why values of $k$ smaller than $1$ can be considered too.
In Tables \ref{Table-time-1-pulse-2}-\ref{Table-time-1-pulse-3}, we provide estimates of the minimal time for $k<1$.
When  $k<0.58$, we did not observe (numerically) convergence towards \textbf{0}.

However, like for the constant releases case, the larger the value of $k$, the lowest the time necessary to enter $[{\bf{0}}, E_1)$. Values of $k$ chosen between $2$ and $5$ seem the most interesting ones.

\begin{table}[H]
  \centering
  \begin{tabular}{|c|c|c|c|c|c|c|c|c|c|c|c|}
\cline{1-4}\cline{6-8}\cline{10-12}
\multicolumn{4}{|c|}{$k=0.58$}&&\multicolumn{3}{|c|}{$k=0.6$}&&\multicolumn{3}{|c|}{$k=0.7$}\\
  \cline{1-4}\cline{6-8}\cline{10-12}
  \backslashbox{$\gamma$}{$\underline{M}_T^*$} & 100 & 500 & 800 && 100 & 500 & 800&& 100 & 500 & 800\\
  \cline{1-4}\cline{6-8}\cline{10-12}
0.04   & 3073 & 3003  & 2827 && 1065 & 997 & 852 && 449 & 388 & 300\\
\cline{1-4}\cline{6-8}\cline{10-12}
0.06   & 3732 &  3696  &  3677 && 1111  & 1075  & 1057  && 467  &  433 & 416 \\
   \cline{1-4}\cline{6-8}\cline{10-12}
0.08   & 4322 & 4294 & 4282 && 1137 & 1109 & 1098 && 477 & 450 & 439\\
\cline{1-4}\cline{6-8}\cline{10-12}
0.1   & $\infty$  &  $\infty$  & $\infty$  && 1215  &  1191 & 1182  && 486  & 462  & 454 \\
   \cline{1-4}\cline{6-8}\cline{10-12}
\end{tabular}
\caption{Periodic impulsive releases are done every 7 days. Numerical
  estimates of the minimal times (in days) to reach $\underline{Y}$, using massive periodic impulsive releases, $M_T^*=
  \Lambda\tau\geq k\times M_{T_1}^{per}$. The symbol $\infty$ denotes that the result is greater than $10^6$.}
  \label{Table-time-1-pulse-2}
\end{table}

\begin{table}[H]
  \centering
  \begin{tabular}{|c|c|c|c|c|c|c|c|c|c|c|c|}
\cline{1-4}\cline{6-8}\cline{10-12}
\multicolumn{4}{|c|}{$k=0.8$}&&\multicolumn{3}{|c|}{$k=0.9$}&&\multicolumn{3}{|c|}{$k=1$}\\
  \cline{1-4}\cline{6-8}\cline{10-12}
  \backslashbox{$\gamma$}{$\underline{M}_T^*$} & 100 & 500 & 800&& 100 & 500 & 800&& 100 & 500 & 800\\
  \cline{1-4}\cline{6-8}\cline{10-12}
0.04   & 335 & 278 & 210 && 282 & 228 & 171 && 250 & 199 & 148 \\
\cline{1-4}\cline{6-8}\cline{10-12}
0.06   & 351 &  318  &  302 && 297  & 265  & 250  && 265  &  234 & 219 \\
   \cline{1-4}\cline{6-8}\cline{10-12}
0.08   & 359 & 333 & 323&& 305 & 279 & 269 && 273 & 247 & 237 \\
   \cline{1-4}\cline{6-8}\cline{10-12}
   0.1   & 366  &  343  & 334  && 311  &  288 & 279  && 278  & 255  & 247 \\
   \cline{1-4}\cline{6-8}\cline{10-12}
\end{tabular}
\caption{Periodic impulsive releases are done every 7 days. Numerical estimates of the minimal times (in days) to reach $\underline{Y}$, using massive periodic impulsive releases, $M_T^*=
  \Lambda\tau\geq k\times M_{T_1}^{per}$}
  \label{Table-time-1-pulse-3}
\end{table}

\subsection{Mechanical control or not?}
In general, using SIT alone is not efficient. It is preferable to consider other bio-control tools. Against mosquito, it has been showed that mechanical control (MC), which consists of removing the breeding sites, can be an additional efficient control tool \cite{Dumont2010,Dumont2008}, and in particular coupled with SIT \cite{Dumont2012}. This is a cheap control, but it requires the support of the local population.

We now assume that the MC leads an increase of $\mu_{A,2}$, that is a decrease of the wild aquatic stage equilibrium $A^*$ (see Table
\ref{Table-Wild-equilibrium} for values of $A^*$).  According to relation (\ref{Wild-equilibria-definition}), page \pageref{Wild-equilibria-definition}, we deduce that reducing $A^*$ for $MC\%$ corresponds to an increase of $\mu_{A,2}$ as follows
\begin{equation}\label{reduction-mecanique}
 \begin{array}{rcl}
\mu_{A,2,MC}&=& \displaystyle\frac{(\gamma+\mu_{A,1})}{(1-\frac{MC}{100})A^*}(\mathcal{R}-1).\\
 \end{array}
\end{equation}
In Table \ref{Table-reduction-lutte-mecanique}, we provide $\mu_{2,A,MC}$ and the wild equilibrium $E^*_{MC}$, for $MC=0$, $20\%$ and $40\%$ in (\ref{reduction-mecanique}).

\begin{table}[H]
  \centering
  \begin{tabular}{|c|c|c|c|c|c|c|c|c|c|c|c|c|}
    \hline
    \multicolumn{5}{|c|}{$MC=0$} &\multicolumn{4}{|c|}{$MC=20$}&\multicolumn{4}{|c|}{$MC=40$} \\
    \hline
    $\mu_{A,2,MC}$  & \multicolumn{4}{|c|}{$2\times 10^{-4}$} &  \multicolumn{4}{|c|}{$2.5\times 10^{-4}$} &  \multicolumn{4}{|c|}{$3.3333\times 10^{-4}$} \\
    \hline
    $\gamma$ & 0.04 & 0.06 & 0.08 & 0.1 & 0.04 & 0.06 & 0.08 & 0.1 & 0.04 & 0.06 & 0.08 & 0.1\\
    \hline
    $A^*$ & 9350 & 14150 & 18950 & 23750  &  7480 & 11320 & 15160 & 19000   &  5610 & 8490 & 11370 & 14250 \\
    \hline
    $M^*$ & 1335 & 3031 & 5412 & 8479     &  1068 & 2425 & 4330 & 6783   &  801 & 1819 & 3247 & 5087 \\
    \hline
    $F^*$ & 1834 & 4160 & 7428 & 11637    &  1466 & 3328 & 5943 & 9310   &  1010 & 2496 & 4457 & 6983 \\
    \hline
  \end{tabular}
  \caption{Impact of MC on the wild equilibrium $E^*_{MC}$}\label{Table-reduction-lutte-mecanique}
\end{table}
Clearly, the impact of MC on the wild equilibrium is quite obvious. However, MC can be  limited in space and time. 

Since the objective of massive SIT release is to enter (rapidly) in $[{\bf{0}},E_1)$, it is also interesting to see the impact of MC treatment on the unstable equilibrium, $E_{1,MC}$, for a given targeted amount of sterile males, $\underline{M}_T^*$. This is summarized in Tables \ref{Table-Positive-equilibrium-MC-1} and \ref{Table-Positive-equilibrium-MC-2}. In fact, and this is a good news, we have $E_{1,MC}>E_1=E_{1,0}$. Thus, with MC, the wild equilibrium, $E^*_{MC}$, decreases and the size of $[{\bf{0}},E_1)$ increases, such that we can expect a good gain in terms of minimal time to enter in  $[{\bf{0}},E_1)$, using massive releases.

\begin{table}[H]
  \centering
  \begin{tabular}{|c|c|c|}
\hline
  \backslashbox{$\gamma$}{$\underline{M}_T^*$} & 100 & 500\\
  \hline
0.04   & (37.39,5.34,0.37)' & (347.57,49.56,6.14)'  \\
  \hline
0.06   & (18.96,4.06,0.22)' & (115.79,24.8,1.6)'   \\
   \hline

0.08   & (12.3,3.51,0.163)' & (68.24,19.49,1)'   \\
   \hline
0.1   & (8.98,3.21,0.137)' & (47.88,17.09,0.78)'   \\
   \hline
\end{tabular}
\caption{Values of $E_{1,MC}$ for different values of the targeted releases amount, $\underline{M}_T^*$, and various values of $\gamma$, when $MC=20\%$. }\label{Table-Positive-equilibrium-MC-1}
\end{table}
\begin{table}[H]
  \centering
  \begin{tabular}{|c|c|c|}
\hline
  \backslashbox{$\gamma$}{$\underline{M}_T^*$} & 100 & 500\\
  \hline
0.04   & (38.82,5.54,0.4)' & (646.33,92.3,19.7)'  \\
  \hline
0.06   & (19.25,4.12,0.22)' & (127.8,24.37,1.9)'   \\
   \hline

0.08   & (12.4,3.54,0.166)' & (71.5,20.4,1.1)'   \\
   \hline
0.1   & (9.03,3.22,0.138)' & (49.2,17.58,0.82)'   \\
   \hline
\end{tabular}
\caption{Values of $E_{1,MC}$ for different values of the targeted releases amount, $\underline{M}_T^*$, and various values of $\gamma$, when $MC=40\%$. }\label{Table-Positive-equilibrium-MC-2}
\end{table}

Minimal time results are given in Tables \ref{Table-time-full-MC}-\ref{Table-time-full-MC-2}, when we consider that MC has started before SIT and goes on once SIT starts. Clearly, the gain in time is "small", indicating that MC does not drastically decay the minimal time to reach $[{\bf{0}},E_1)$.
\begin{table}[H]
  \centering
 \begin{tabular}{|c|c|c|c|c|c|c|c|c|}
\hline
\multicolumn{9}{|c|}{The case of continuous and constant release}\\
\hline
 \cline{1-3}\cline{5-6}\cline{8-9}
\multicolumn{3}{|c|}{$k=2$}&&\multicolumn{2}{|c|}{$k=5$}&&\multicolumn{2}{|c|}{$k=10$}\\
  \cline{1-3}\cline{5-6}\cline{8-9}
  \backslashbox{$\gamma$}{$\underline{M}_T^*$} & 100 & 500  && 100 & 500 && 100 & 500  \\
   \cline{1-3}\cline{5-6}\cline{8-9}
0.04   & 213(4) & 155(14) && 137(4) & 93(10) && 120(3) & 80(8) \\
    \cline{1-3}\cline{5-6}\cline{8-9}
0.06   &  228(4) & 195(6)  && 152(3)  &  123(5) && 134(3)  & 107(4) \\
   \cline{1-3}\cline{5-6}\cline{8-9}
0.08  & 236(3) & 210(4) && 160(2) & 136(3) && 141(3) & 118(4) \\
    \cline{1-3}\cline{5-6}\cline{8-9}
0.1   & 241(3)  & 218(3)  && 164(3)  &  143(3) && 146(3)  & 125(4) \\
    \cline{1-3}\cline{5-6}\cline{8-9}
   \hline
\multicolumn{9}{|c|}{The case of periodic pulse release}\\
\hline
 \cline{1-3}\cline{5-6}\cline{8-9}
\multicolumn{3}{|c|}{$k=2$}&&\multicolumn{2}{|c|}{$k=5$}&&\multicolumn{2}{|c|}{$k=10$}\\
   \cline{1-3}\cline{5-6}\cline{8-9}
  \backslashbox{$\gamma$}{$\underline{M}_T^*$} & 100 & 500  && 100 & 500 && 100 & 500 \\
   \cline{1-3}\cline{5-6}\cline{8-9}
0.04   & 157(9) & 109(11) && 123(4) & 82(9)  && 113(4) & 75(8) \\
   \cline{1-3}\cline{5-6}\cline{8-9}
   0.06   & 172(3)  & 142(5)  && 137(3)  & 110(4)  && 127(3)  & 101(4) \\
   \cline{1-3}\cline{5-6}\cline{8-9}
0.08  & 180(3) & 155(4) && 145(3) & 122(3) && 135(3) & 112(4) \\
   \cline{1-3}\cline{5-6}\cline{8-9}
0.1   & 185(2)  & 162(4)  && 150(2)  &  129(3)  && 140(2)  & 119(3) \\
   \cline{1-3}\cline{5-6}\cline{8-9}
\hline
\end{tabular}
\caption{The case when $20\%$  of MC takes place all over the time. Numerical estimates of the minimal times (in days) to reach $\underline{Y}$, using massive releases, $M_T^*=k\times
  M_{T_1}$.  The values in the brackets indicate the gain in days compared to SIT alone.}
  \label{Table-time-full-MC}
\end{table}

\begin{table}[H]
  \centering
\begin{tabular}{|c|c|c|c|c|c|c|c|c|}
\hline
\multicolumn{9}{|c|}{The case of continuous and constant release}\\
\hline
 \cline{1-3}\cline{5-6}\cline{8-9}
\multicolumn{3}{|c|}{$k=2$}&&\multicolumn{2}{|c|}{$k=5$}&&\multicolumn{2}{|c|}{$k=10$}\\
  \cline{1-3}\cline{5-6}\cline{8-9}
  \backslashbox{$\gamma$}{$\underline{M}_T^*$} & 100 & 500  && 100 & 500 && 100 & 500  \\
   \cline{1-3}\cline{5-6}\cline{8-9}
0.04   & 206(11) & 116(53) && 132(9) & 70(33) && 114(10) & 60(28) \\
    \cline{1-3}\cline{5-6}\cline{8-9}
0.06   &  224(8) & 186(15)  && 147(8)  &  116(12) && 129(8)  & 100(11) \\
   \cline{1-3}\cline{5-6}\cline{8-9}
0.08  & 232(7) & 204(10) && 156(6) & 130(9) && 138(6) & 114(8) \\
    \cline{1-3}\cline{5-6}\cline{8-9}
0.1   & 237(7)  & 213(8)  && 161(6)  &  138(8) && 143(6)  & 121(8) \\
    \cline{1-3}\cline{5-6}\cline{8-9}
   \hline
\multicolumn{9}{|c|}{The case of periodic pulse release}\\
\hline
 \cline{1-3}\cline{5-6}\cline{8-9}
\multicolumn{3}{|c|}{$k=2$}&&\multicolumn{2}{|c|}{$k=5$}&&\multicolumn{2}{|c|}{$k=10$}\\
   \cline{1-3}\cline{5-6}\cline{8-9}
  \backslashbox{$\gamma$}{$\underline{M}_T^*$} & 100 & 500  && 100 & 500 && 100 & 500 \\
   \cline{1-3}\cline{5-6}\cline{8-9}
0.04   & 151(15) & 82(38) && 118(9) & 62(29)  && 108(9) & 57(26) \\
   \cline{1-3}\cline{5-6}\cline{8-9}
   0.06   & 167(8)  & 134(13)  && 133(7)  & 103(11)  && 123(7)  & 94(11) \\
   \cline{1-3}\cline{5-6}\cline{8-9}
0.08  & 176(7) & 149(10) && 141(7) & 117(8) && 131(7) & 107(9) \\
   \cline{1-3}\cline{5-6}\cline{8-9}
0.1   & 181(6)  & 158(8)  && 146(6)  &  125(7)  && 136(6)  & 115(7) \\
   \cline{1-3}\cline{5-6}\cline{8-9}
\hline
\end{tabular}
\caption{The case when $40\%$  of MC takes place all over the time. Numerical estimates of the minimal times (in days) to reach $\underline{Y}$, using massive releases, $M_T^*=k\times
  M_{T_1}$.  The values in the brackets indicate the gain in days compared to SIT alone.}
 \label{Table-time-full-MC-2}
\end{table}

MC is a useful tool. However, to be really efficient, whatever the type of releases, MC needs to reduce the potential breeding site by $40\%$.   

In fact, the combination of control strategies needs to be considered according to the location. In la R\'eunion, a french overseas department in the Indian Ocean where a SIT project is ongoing, there is a seasonal effect on the wild mosquito population \cite{dufourd2013}, such that the best period to start SIT is in September, when the size of the wild mosquito   population is low. In general there is a factor $10$ in the population estimates between the wet season (February-March) and the dry season (September) (see for instance \cite{legoff2019}). In Cali (Colombia), there is no seasonal effect, such that the wild population is more or less constant along the year. In order to use the SIT in an efficient manner in Cali, a population reduction is necessary. 

One possible way, and also recommended by IAEA (the International Atomic Energy Agency) for SIT control, is to first use insecticide to reduce the population by a factor $5$ or $10$, and then to use SIT control. This is what we consider now: during one week, before SIT starts, we combine MC and an adulticide treatment, assuming $100\%$ efficiency. 

In Tables \ref{Table-wild-equilibrium-MC-altudice} and \ref{Table-wild-equilibrium-MC-altudice-2}, we provide the values obtained after one week of adulticide treatment without and with MC.

\begin{table}[H]
  \centering
  \begin{tabular}{|c|c|c|c|c|}
    \hline
    \multicolumn{5}{|c|}{Adulticide during one week} \\
    \hline
   \multicolumn{5}{|c|}{$MC=0$} \\
   \hline
    $\gamma$ & 0.04 & 0.06 & 0.08 & 0.1 \\
    \hline
    $A_7$ & 1897.9 & 2645.1 & 3387.1 & 4114 \\
    \hline
    $M_7$ & 46.2 & 98.3 & 169.5 & 258.6 \\
    \hline
    $F_7$ & 49.3 & 105.4 & 182.2 & 278.2 \\
    \hline
  \end{tabular}
  \caption{Solution ($A_7$, $M_7$,$F_7$)' of the model after one week of adulticide treatment \tcm{only}}\label{Table-wild-equilibrium-MC-altudice}
\end{table}

\begin{table}[H]
  \centering
\begin{tabular}{|c|c|c|c|c|c|c|c|c|}
    \hline
    \multicolumn{9}{|c|}{Adulticide during one week} \\
    \hline
   \multicolumn{5}{|c|}{$MC=20$}&\multicolumn{4}{|c|}{$MC=40$} \\
   \hline
    $\gamma$ &  0.04 & 0.06 & 0.08 & 0.1 & 0.04 & 0.06 & 0.08 & 0.1\\
    \hline
    $A_7$ &   1518.4  &  2116 & 2709.7 & 3291.2 & 1138.9   &  1587.2 & 2032.5 & 2468.6 \\
    \hline
    $M_7$ &   37 & 78.6 & 135.6 & 206.9   &  27.7 & 59 & 101.7 & 155.2 \\
    \hline
    $F_7$ &  39.5 & 84.3 & 145.7 & 222.5   &  29.6 & 63.2 & 109.3 & 166.9 \\
    \hline
  \end{tabular}
  \caption{Solution ($A_7$, $M_7$,$F_7$)' of the model after one week of  adulticide treatment combined with MC.}
  \label{Table-wild-equilibrium-MC-altudice-2}
\end{table}

Clearly, according to the tables above, after one weak of adulticide treatment, the size of the mosquito population has been drastically reduced, such that the SIT treatment will now start at the point $X_7=(A_7,M_7,F_7)'$. That is why an impact on the minimal time to enter the basin $[{\bf{0}},E_{1,MC})$ is expected.

Indeed, Table \ref{Table-time-Adulticide-Only}, page \pageref{Table-time-Adulticide-Only}, clearly confirms that the gain in the entry time is rather important for the adulticide treatment only: it ranges from $35$ to $95$ days.

\begin{table}[H]
  \centering
\begin{tabular}{|c|c|c|c|c|c|c|c|c|}
\hline
\multicolumn{9}{|c|}{The case of continuous and constant release}\\
\hline
 \cline{1-3}\cline{5-6}\cline{8-9}
\multicolumn{3}{|c|}{$k=2$}&&\multicolumn{2}{|c|}{$k=5$}&&\multicolumn{2}{|c|}{$k=10$}\\
  \cline{1-3}\cline{5-6}\cline{8-9}
  \backslashbox{$\gamma$}{$\underline{M}_T^*$} & 100 & 500  && 100 & 500 && 100 & 500  \\
   \cline{1-3}\cline{5-6}\cline{8-9}
0.04   & 122(95) & 74(95) && 92(49) & 54(49) && 85(38) & 50(38) \\
    \cline{1-3}\cline{5-6}\cline{8-9}
0.06   &  137(95) & 107(94)  && 106(49)  &  79(49) && 98(39)  & 72(39) \\
   \cline{1-3}\cline{5-6}\cline{8-9}
0.08  & 146(93) & 121(93) && 113(49) & 90(49) && 105(39) & 83(39) \\
    \cline{1-3}\cline{5-6}\cline{8-9}
0.1   & 150(94)  & 128(93)  && 118(49)  &  97(49) && 110(39)  & 90(39) \\
    \cline{1-3}\cline{5-6}\cline{8-9}
   \hline
\multicolumn{9}{|c|}{The case of periodic pulse release}\\
\hline
 \cline{1-3}\cline{5-6}\cline{8-9}
\multicolumn{3}{|c|}{$k=2$}&&\multicolumn{2}{|c|}{$k=5$}&&\multicolumn{2}{|c|}{$k=10$}\\
   \cline{1-3}\cline{5-6}\cline{8-9}
  \backslashbox{$\gamma$}{$\underline{M}_T^*$} & 100 & 500  && 100 & 500 && 100 & 500 \\
   \cline{1-3}\cline{5-6}\cline{8-9}
0.04   & 101(65) & 60(60) && 87(40) & 51(40)  && 82(35) & 48(35) \\
   \cline{1-3}\cline{5-6}\cline{8-9}
   0.06   & 115(60)  & 87(60)  && 100(40)  & 74(40)  && 95(35)  & 70(35) \\
   \cline{1-3}\cline{5-6}\cline{8-9}
0.08  & 123(60) & 99(60) && 107(41) & 85(40) && 103(35) & 81(35) \\
   \cline{1-3}\cline{5-6}\cline{8-9}
0.1   & 127(60)  & 106(60)  && 112(40)  &  91(41)  && 107(35)  & 87(38) \\
   \cline{1-3}\cline{5-6}\cline{8-9}
\hline
\end{tabular}
\caption{Numerical estimates of the minimal times (in days) to reach $\underline{Y}$, using massive releases, $M_T^*=k\times
  M_{T_1}$.  The values in the brackets indicate the gain in days compared to SIT alone.}\label{Table-time-Adulticide-Only}
\end{table}

In Tables \ref{Table-time-MC-Adulticide-stop-1} and  \ref{Table-time-MC-Adulticide-stop-2}, we present the results when MC is combined with the adulticide treatment. As, expected, the results are improved. However, the gain, compared to the adulticide treatment alone is small, such that the best combination would be "adulticide treatment for seven days, followed by permanent SIT treatment".

\begin{table}[H]
  \centering
\begin{tabular}{|c|c|c|c|c|c|c|c|c|}
\hline
\multicolumn{9}{|c|}{The case of continuous and constant release}\\
\hline
 \cline{1-3}\cline{5-6}\cline{8-9}
\multicolumn{3}{|c|}{$k=2$}&&\multicolumn{2}{|c|}{$k=5$}&&\multicolumn{2}{|c|}{$k=10$}\\
  \cline{1-3}\cline{5-6}\cline{8-9}
  \backslashbox{$\gamma$}{$\underline{M}_T^*$} & 100 & 500  && 100 & 500 && 100 & 500  \\
   \cline{1-3}\cline{5-6}\cline{8-9}
0.04   & 116(101) & 68(101) && 89(52) & 50(53) && 82(41) & 47(41) \\
    \cline{1-3}\cline{5-6}\cline{8-9}
0.06   &  131(101) & 101(100)  && 102(53)  &  75(53) && 95(42)  & 70(41) \\
   \cline{1-3}\cline{5-6}\cline{8-9}
0.08  & 139(100) & 114(100) && 110(52) & 87(52) && 103(41) & 80(42) \\
    \cline{1-3}\cline{5-6}\cline{8-9}
0.1   & 144(100)  & 122(99)  && 115(52)  &  94(52) && 107(42)  & 87(42) \\
    \cline{1-3}\cline{5-6}\cline{8-9}
   \hline
\multicolumn{9}{|c|}{The case of periodic pulse release}\\
\hline
 \cline{1-3}\cline{5-6}\cline{8-9}
\multicolumn{3}{|c|}{$k=2$}&&\multicolumn{2}{|c|}{$k=5$}&&\multicolumn{2}{|c|}{$k=10$}\\
   \cline{1-3}\cline{5-6}\cline{8-9}
  \backslashbox{$\gamma$}{$\underline{M}_T^*$} & 100 & 500  && 100 & 500 && 100 & 500 \\
   \cline{1-3}\cline{5-6}\cline{8-9}
0.04   & 97(69) & 56(64) && 84(43) & 48(43)  && 80(37) & 46(37) \\
   \cline{1-3}\cline{5-6}\cline{8-9}
   0.06   & 111(64)  & 83(64)  && 97(43)  & 71(43)  && 93(37)  & 68(37) \\
   \cline{1-3}\cline{5-6}\cline{8-9}
0.08  & 119(64) & 95(64) && 104(44) & 82(43) && 100(38) & 78(38) \\
   \cline{1-3}\cline{5-6}\cline{8-9}
0.1   & 124(63)  & 102(64)  && 109(43)  &  89(43)  && 105(37)  & 85(37) \\
   \cline{1-3}\cline{5-6}\cline{8-9}
\hline
\end{tabular}
\caption{Combination of adulticide and $20\%$ of MC, followed by SIT. Numerical estimates of the minimal times (in days) to reach $\underline{Y}$, using massive releases, $M_T^*=k\times M_{T_1}$.  The values in the brackets indicate the gain in days compared to SIT alone.}\label{Table-time-MC-Adulticide-stop-1}
\end{table}

\begin{table}[H]
  \centering
\begin{tabular}{|c|c|c|c|c|c|c|c|c|}
\hline
\multicolumn{9}{|c|}{The case of continuous and constant release}\\
\hline
 \cline{1-3}\cline{5-6}\cline{8-9}
\multicolumn{3}{|c|}{$k=2$}&&\multicolumn{2}{|c|}{$k=5$}&&\multicolumn{2}{|c|}{$k=10$}\\
  \cline{1-3}\cline{5-6}\cline{8-9}
  \backslashbox{$\gamma$}{$\underline{M}_T^*$} & 100 & 500  && 100 & 500 && 100 & 500  \\
   \cline{1-3}\cline{5-6}\cline{8-9}
0.04   & 107(110) & 59(110) && 85(56) & 46(57) && 79(44) & 43(45) \\
    \cline{1-3}\cline{5-6}\cline{8-9}
0.06   &  123(109) & 93(108)  && 98(57)  &  71(57) && 92(45)  & 66(45) \\
   \cline{1-3}\cline{5-6}\cline{8-9}
0.08  & 132(107) & 107(107) && 106(56) & 83(56) && 99(45) & 77(45) \\
    \cline{1-3}\cline{5-6}\cline{8-9}
0.1   & 137(107)  & 114(107)  && 110(57)  &  90(56) && 104(45)  & 84(45) \\
    \cline{1-3}\cline{5-6}\cline{8-9}
   \hline
\multicolumn{9}{|c|}{The case of periodic pulse release}\\
\hline
 \cline{1-3}\cline{5-6}\cline{8-9}
\multicolumn{3}{|c|}{$k=2$}&&\multicolumn{2}{|c|}{$k=5$}&&\multicolumn{2}{|c|}{$k=10$}\\
   \cline{1-3}\cline{5-6}\cline{8-9}
  \backslashbox{$\gamma$}{$\underline{M}_T^*$} & 100 & 500  && 100 & 500 && 100 & 500 \\
   \cline{1-3}\cline{5-6}\cline{8-9}
0.04   & 92(74) & 50(70) && 80(47) & 44(47)  && 77(40) & 42(41) \\
   \cline{1-3}\cline{5-6}\cline{8-9}
   0.06   & 106(69)  & 78(69)  && 93(47)  & 67(47)  && 90(40)  & 64(41) \\
   \cline{1-3}\cline{5-6}\cline{8-9}
0.08  & 114(69) & 90(69) && 101(47) & 78(47) && 97(41) & 75(41) \\
   \cline{1-3}\cline{5-6}\cline{8-9}
0.1   & 118(69)  & 97(69)  && 105(47)  &  85(47)  && 102(40)  & 81(41) \\
   \cline{1-3}\cline{5-6}\cline{8-9}
\hline
\end{tabular}
\caption{Combination of adulticide and $40\%$ of MC, followed by SIT. Numerical estimates of the minimal times (in days) to reach $\underline{Y}$, using massive releases, $M_T^*=k\times
  M_{T_1}$.  The values in the brackets indicate the gain in days compared to SIT alone.}\label{Table-time-MC-Adulticide-stop-2}
\end{table}

\section{Conclusion}\label{conclusion}
In this manuscript we complete the work done in \cite{Strugarek2019}, when no Allee effect can be exploit in the SIT treatment strategy. We know that without Allee effect, a permanent SIT treatment is necessary to maintain the population under a certain threshold and eventually to drive it to extinction, at a minimal cost. Last but not least, we also prove that reducing the size of the mosquito population using adulticide or using a seasonality effect, really improve the efficiency of SIT, and, in addition its cost. Mechanical control, while useful to reduce the wild population when no other treatments are available, can help to improve SIT, but, in fact, its impact seems to be limited, such that, in terms of cost, it seems not to be necessary or useful. In addition, we know that mechanical control is difficult to maintain, and eventually can have a negative effect \cite{DT2016,TD2019}.

Based on this work, several further extensions are possible: couple the SIT model with an epidemiological model, to derive, for instance, the (epidemiological) threshold value to reach, in order to reduce the epidemiological risk, as it was done in \cite{Dumont2012}; take into account the costs of the different treatments in order to derive the most interesting combinations from the economical point of view; another enhancement would be to take into account the spatial component. Last but not least, comparison and links with SIT field experiments, against mosquitoes or fruit flies, are needed to enhance the model and also SIT control strategies.

\vspace{0.5cm}
\section*{Acknowledgments}
This study is part of the Phase 2B {\color{black} "SIT feasibility project against \textit{Aedes albopictus} in Reunion Island"}, jointly funded by the French Ministry of Health and the European Regional Development Fund (ERDF). All authors were (partially) supported by the DST/NRF SARChI Chair  in Mathematical Models and Methods in Biosciences and Bioengineering at the University of Pretoria (grant 82770). YD is also partially supported by the GEMDOTIS project, funded by the call ECOPHYTO 2018 (Action 27).

\bibliographystyle{plain}
\bibliography{bibliography}

\end{document}